\documentclass[a4paper,11pt]{amsart}
\usepackage{amssymb,amsfonts}
\usepackage{amsmath,amsthm}
\usepackage[english]{babel}
\usepackage{amscd}
\usepackage{pgf,tikz}
\usepackage[utf8]{inputenc}
\usepackage{comment}
\usepackage{mathtools}
\usepackage[left=2cm,right=2cm, top=3cm, bottom=3cm]{geometry}

\mathtoolsset{showonlyrefs}
\usepackage[all,cmtip]{xy}
\usepackage{enumerate}
\xyoption{all}
\usepackage{srcltx}
\usepackage{textcomp}
\usepackage{parskip}
\usepackage{amsthm, amsmath, amssymb,latexsym}
\usepackage{amsfonts,epsfig,amscd}
\usepackage[]{fontenc}
\usepackage{xy}
\usepackage{enumerate}
\usepackage[]{fontenc}
\usepackage[all]{xy}
\newtheorem{theorem}{Theorem}[section]
\newtheorem{lemma}[theorem]{Lemma}
\newtheorem{corollary}[theorem]{Corollary}
\newtheorem{proposition}[theorem]{Proposition}
\newtheorem{remark}[theorem]{Remark}
\newtheorem{definition}[theorem]{Definition}
\newtheorem{example}[theorem]{Example}

\newtheorem{conjecture}[theorem]{Conjecture}

\newcommand{\ra}{\rightarrow}

\newcommand{\longra}{\longrightarrow}

\newcommand{\longlra}{\longleftrightarrow}

\newcommand{\hra}{\hookrightarrow}

\newcommand{\thra}{\twoheadrightarrow}

\newcommand{\beq}{\begin{equation}}
\newcommand{\enq}{\end{equation}}
\newcommand{\beqn}{\begin{equation}}
\newcommand{\enqn}{\end{equation}}

\newcommand{\caA}{\mathcal{A}}

\newcommand{\caC}{\mathcal{C}}
\newcommand{\caE}{\mathcal{E}}
\newcommand{\caF}{\mathcal{F}}

\newcommand{\caM}{\mathcal{M}}

\newcommand{\caO}{\mathcal{O}}

\newcommand{\mC}{\mathbb{C}}

\newcommand{\mP}{\mathbb{P}}

\newcommand{\mZ}{\mathbb{Z}}

\DeclareMathOperator{\tr}{\rm tr}

\DeclareMathOperator{\Ext}{\rm Ext}
\DeclareMathOperator{\Sym}{\rm Sym}

\DeclareMathOperator{\im}{\rm im}
\DeclareMathOperator{\coker}{\rm coker}
\DeclareMathOperator{\rk}{\rm rk}

\DeclareMathOperator{\ord}{\rm ord}

\DeclareMathOperator{\Pic}{\rm Pic}

\DeclareMathOperator{\Spec}{\rm Spec}

\newcommand{\ToDo}[1]{{\textcolor{red}{#1}}}

\title{General infinitesimal variations of Hodge structure of ample curves in surfaces}
\makeatletter

\@addtoreset{equation}{section}
\makeatother

\author{Víctor González-Alonso}
\address{Víctor González-Alonso,
Institut für Algebraische Geometrie,
Leibniz Universit\"at Hannover
Welfengarten 1,
30167 Hannover, Germany}
\email{gonzalez@math.uni-hannover.de}

\author{Sara Torelli}
\address{Sara Torelli,
Università di Torino,
Via Verdi 8, 10125 Torino, Italia}
\email{sara.torelli7@gmail.com}
\date{\today}

\thanks{The project is partially supported by the Alexander von Humboldt Foundation.\\ Data sharing not applicable to this article as no datasets were generated or analysed during the current study. None of the authors has any kind of conflict of interest.}

\begin{document}

\begin{abstract}
Given a smooth projective complex curve inside a smooth projective surface, one can ask how its Hodge structure varies when the curve moves inside the surface. In this paper, we develop a general theory to study the infinitesimal version of this question in the case of ample curves. We can then apply the machinery to show that the infinitesimal variation of Hodge structure of a general deformation of an ample curve in $\mP^1\times\mP^1$ is an isomorphism.
\end{abstract}
\maketitle

\bigskip

\section{Introduction}

Let $\pi\colon \caC\to B$ be a family of (smooth projective) complex curves of genus $g$ over a smooth complex manifold $B$, which we consider as a deformation of a fibre $C=\pi^{-1}\left(0\right)$ over a distinguished point $0\in B$. A tangent vector $v\in T_{B,0}$ induces a first-order infinitesimal deformation of $C$, whose Kodaira-Spencer class we denote by $\xi_v\in H^1\left(C,T_C\right)$. Recall that $H^1\left(C,T_C\right)$ is naturally identified with the tangent space of the moduli space (or stack) $\caM_g$ at $\left[C\right]$, hence we can think of elements in $H^1\left(C,T_C\right)$ as tangent directions in $\caM_g$.

The infinitesimal variation of Hodge structure (IVHS) associated with the family $\pi\colon\caC\to B$ is given by a linear map
$$\theta\colon T_{B,0}\otimes H^0\left(C,\Omega_C^1\right)\to H^1\left(C,\caO_C\right),$$
where
\begin{equation} \label{eq:general-Higgs}
\theta\left(v\otimes -\right)\colon H^0\left(C,\Omega_C^1\right)\to H^1\left(C,\caO_C\right)
\end{equation}
is given (up to non-zero scalar) by the product with the Kodaira-Spencer class $\xi_v$. The kernel of this product map consists of the holomorphic $1$-forms on $C$ that extend to the first-order infinitesimal neighbourhood of $C\subseteq\caC$ (in the direction $v$), which can thus be considered as ``infinitesimally constant'' or invariant in the variation of Hodge structure.

It is known (\cite{LP-product-curves, GT-KU}) that, for any curve $C$, the product with a general $\xi\in H^1\left(C,T_C\right)$ induces an isomorphism $H^0\left(C,\Omega_C^1\right)\to H^1\left(C,\caO_C\right)$. Hence, a general deformation of $C$ has maximal variation of Hodge structure, since the nearby fibres have no common holomorphic $1$-form. We have recently used this fact to prove that, for $g\geq 4$, the Torelli map $\caM_g\to \caA_g$ (as well as other natural modular maps) is infinitesimally rigid as a morphism of smooth DM-stacks (\cite{CGT-Rigidity}).

The above results consider general deformations of $C$, but in many cases it is useful to restrict the kind of deformations to those preserving certain geometric features of $C$. In particular, it is natural to consider curves lying in a particular surface $S$. This set-up has important applications to the study of moduli spaces of curves $\mathcal{M}_g$ in low genus, so that a general curve sits inside a given surface. For instance, the general curve of genus $3$ is a smooth quartic and therefore properties of the moduli space of genus $3$ curves can be given by studying them as plane curves (see for example \cite{GK06}). A very important example is given by curves contained in $K3$ surfaces, as a general curve of genus $g\leq 9$ or equal to $11$ is contained in a K3 surface, and many results about the moduli space $\caM_g$ are proved by using related properties (see for example \cite{Lazarsfeld:BN}, \cite{Mukai:M11}, \cite{Voisin:Green-conj-even}, \cite{Voisin:Green-conj-odd}).

The important point for our study is that, as soon as the genus grows, the general curve of genus $g$ is not expected to be sitting inside a prescribed surface. Therefore, it might very well happen that the general deformations \emph{within these restrictions} no longer induce isomorphisms in \eqref{eq:general-Higgs}. Results in this direction provide geometric information on the tangent bundle of the corresponding sublocus of the moduli space of curves $\mathcal{M}_g$. Later in the paper, we concentrate on the case of curves in $\mathbb{P}^1\times \mathbb{P}^1$,  which is related to the study of the $k-$gonal locus in $\mathcal{M}_g$. Indeed, curves admitting a $g^1_k$ and a $g^1_l$, with coprime $k,l$, can be thought of as curves in $\mathbb{P}^1\times \mathbb{P}^1$ of bidegree $(k,l)$.

In this paper, we address the case of deformations of an \emph{ample} curve $C$ inside a given fixed surface $S$. If we denote by $L=\caO_S\left(C\right)$ the associated (ample) line bundle on $S$, then we consider only (infinitesimal) \emph{embedded} deformations of $C$ in $S$ within the linear system $\left|L\right|$. If $\sigma\in H^0\left(S,L\right)$ is a section vanishing precisely on $C$, such deformations are given by
$$\caC=\left\{\sigma+\epsilon\tau=0\right\}\subseteq S\times\Spec\left(\mC\left[\epsilon\right]/\left(\epsilon^2\right)\right),$$
where $\tau\in H^0\left(S,L\right)$ is another section of $L$ and as usual $\epsilon$ denotes the parameter of the infinitesimal deformation.

The problem is then to analyze the kernel (or the rank) of the produt map $\xi_{\tau}\cdot\colon H^0\left(C,\Omega_C^1\right)\to H^1\left(C,\caO_C\right)$, where $\xi_{\tau}\in H^1\left(T_C\right)$ denotes the associated Kodaira-Spencer class.

Note that there is a natural restriction map $H^0\left(S,\Omega_S^1\right)\to H^0\left(C,\Omega_C^1\right)$, which is injective if $C$ is ample (Proposition \ref{prop:R-K-L-dim-n}). The $1$-forms on $C$ arising this way will be trivially preserved in any deformation of $C$ in $S$. Hence $\ker\left(\xi_{\tau}\cdot\right)$ will always contain a $q\left(S\right)$-dimensional subspace (where $q\left(S\right)=h^1\left(S,\caO_S\right)=h^0\left(S,\Omega_S^1\right)$ is the irregularity of the ambient surface).

This motivates the following definition:

\begin{definition} \label{df:Maximal-embedded-IVHS}
Using the above notations, we say:
\begin{enumerate}
\item $\tau\in H^0\left(S,L\right)$ has {\em maximal (embedded) infinitesimal variation}, if the cup-product map with $\xi_{\tau}\in H^1\left(C,T_C\right)$ has the maximal possible rank, i.e.
$$\rk\left(\xi_{\tau}\cdot\colon H^0\left(C,\Omega_C^1\right)\to H^1\left(C,\caO_C\right)\right)=g\left(C\right)-q\left(S\right),$$
\item the curve $C\subseteq S$ has {\em maximal (embedded) infinitesimal variation}, if there is a $\tau\in H^0\left(S,L\right)$ with maximal infinitesimal variation, and
\item the (ample) line bundle $L$ has {\em maximal infinitesimal variation} if every smooth member $C\in\left|L\right|$ has maximal infinitesimal variation.

If only a general $C\in\left|L\right|$ has maximal infinitesimal variation, we say that $L$ has \emph{generically} maximal infinitesimal variation.
\end{enumerate}
\end{definition}

\begin{remark}
The terminology ``family of maximal variation'' has been used since Viehweg \cite{Viehweg:Weak-positivity-1} with a different meaning. For families of smooth curves, and using the language of moduli spaces, Viehweg's definition means that the map from the base of the family to $\caM_g$ has image of dimension $\dim B$. This is the meaning used, for example, by Y. Dutta and D. Huybrechts in \cite{DH-maximal-variation-K3} regarding linear systems of curves in K3 surfaces. For them, a line bundle $L$ has ``maximal variation'' if the rational map $\left|L\right|\dashrightarrow\caM_g$ mapping a smooth member of $\left|L\right|$ to its isomorphism class is generically of finite degree, i.e. has image of dimension $\dim\left|L\right|$. But it could still happen that all smooth curves in $\left|L\right|$ have a non-trivial common Hodge-substructure. Thus, our notion of ``maximal variation'' is much stronger. In order to avoid confusion, we have added the word ``infinitesimal'', which should remind us that we are considering the infinitesimal variations of  Hodge structure.
\end{remark}

The question we would like to answer is whether, on any given surface $S$, \emph{sufficiently ample} line bundles $L$ have maximal infinitesimal variation. Recall from \cite{Green-hypersurfaces} that a property holds for \emph{sufficiently ample} line bundles if there is an ample line bunlde $L_0$ such that the property holds for any line bundle $L$ such that $L\otimes L_0^{\vee}$ is ample (i.e. $L$ is ``more ample'' than $L_0$). With this language, our conjecture can be stated as follows:

\begin{conjecture} \label{main-conj}
Let $S$ be a projective surface. Then any sufficiently ample line bundle $L$ has maximal infinitesimal variation.
\end{conjecture}

To our best knowledge, this problem has only been recently addressed in \cite{FP-plane-curves}, where they prove the conjecture for plane curves of degree $d\geq 3$, i.e. for all $L=\caO_{\mP^2}\left(d\right)$ with $d\geq 3$ (thus $L_0=\caO_{\mP^2}\left(2\right)$).

The main tool in Favale-Pirola's proof is Macaulay's theorem on the structure of the Jacobian ring of the plane curve. In this paper, we exploit and further develop Green's generalization of Jacobian ideals and rings for smooth hypersurfaces (\cite{Green-hypersurfaces}) to address the problem for arbitrary projective surfaces, on which some of the nice properties of Jacobian rings of plane curves no longer hold. We can then apply our machinery to prove Conjecture \ref{main-conj} for $S=\mP^1\times\mP^1$.

\begin{theorem}[Theorem \ref{thm:IVHS-maximal-P1-P1}] \label{main-thm}
For any $d,e\geq 1$, $L=\caO_X\left(d,e\right)$ has maximal infinitesimal variation at every smooth $Z\in\left|L\right|$. In other words, a general embedded deformation of any smooth ample curve $Z\subseteq\mP^1\times\mP^1$ has maximal infinitesimal variation.
\end{theorem}

It is natural to ask, up to which extent our proof can be extended to more general surfaces. On the one hand, an important step in our proof relies on the second statement of Theorem \ref{thm:J-d-e}, which can be understood as a generalization of the fact that the classical polar map of plane curves is birational onto the image. An equivalent statement for a curve $C$ in a surface $S$ would be that the linear subsystem of $\left|\caO_C\left(C\right)\right|$ given by the image of the natural restriction map
$$H^0\left(S,T_S\right)\to H^0\left(C,N_{C/S}\right)=H^0\left(C,\caO_C\left(C\right)\right)$$
has no base points and induces a birational map $C\to\mP^r$ (Remark \ref{rmk:thm-upper-bound-more-general}). This could hold if $T_S$ is positive enough, and it would be interesting to characterize for which sufaces it holds.

On the other hand, the whole strategy of the proof relies on a special property of the $1$-forms $\alpha\in H^0\left(C,\Omega_C^1\right)$ not arising as restrictions from $H^0\left(S,\Omega_S^1\right)$ but lying on the kernel of a general $\xi_{\tau}$ (see Lemma \ref{lem:alpha2}), and showing that $\alpha=0$ is the only possibility (under certain appropriate hypothesis). However, this cannot be the case if $K_S$ is positive enough (see Remark \ref{rmk:approach-quadrics}), and somehow limits the applicability of our strategy to surfaces with $p_g\left(S\right)\leq 1$, such as rational, $K3$ or abelian surfaces.

In any case, some explicit computations for curves in surfaces $S\subseteq\mP^3$ of general type (degrees $5$ and $6$) make us believe that Conjecture \ref{main-conj} actually holds for every surface.

The paper is structured as follows. In Section \ref{sect:Gen-Jac} we develop the general technical framework to study variations of Hodge structures of smooth hypersurfaces. More precisely, we expand the theory of generalized Jacobian ideals and rings introduced by Green in \cite{Green-hypersurfaces}. In Section \ref{sect:curves-surfaces} we focus on the case of curves on surfaces, proving the crucial Lemma \ref{lem:alpha2} and Theorem \ref{thm:lower-bound-dim-D}. The latter gives a numerical condition satisfied when the curve does not have maximal infinitesimal variation. Finally, in Section \ref{sect:P1-P1} we study in detail the case of curves in $\mP^1\times\mP^1$. We first prove that the generalized Jacobian ideals indeed coincide with the naive definition that one could make using the partial derivatives of a bihomogeneous equation. Secondly, we study the maps induced by the generalized Jacobian ideals in low degrees (analogous to the polar map of a curve in $\mP^2$). We finally prove Theorem \ref{thm:upper-bound-dim-D-P1xP1}, which contradicts the numerical condition obtained in Theorem \ref{thm:lower-bound-dim-D}, and finishes the proof of Theorem \ref{main-thm}.

\textbf{Acknowledgements:} We are particularly grateful to Gian Pietro Pirola, for some very insightful conversations we shared during his visit at Leibniz Universität Hannover. We also thank the referees for their useful comments on the previous version of the paper.

\section{Generalized Jacobian ideals and rings}

\label{sect:Gen-Jac}

The main technical tool to prove Theorem \ref{main-thm} will be generalized Jacobian ideals and rings, which were introduced by Green in \cite{Green-hypersurfaces}, in order to prove the infinitesimal Torelli theorem for sufficiently ample smooth hypersurfaces, as well as to study the local structure of the associated period maps. We will actually need some more precise results than those developed by Green, since his main assumption is that the line bundle $L$ is sufficiently ample, which is mostly used to ensure the vanishing of enough cohomology groups. But we will actually need a better control of which cohomology groups need to vanish for the statements to hold, and also to understand up to what extent these vanishings might fail. For example, Lemma \ref{lem:top-isomorphism} is the first assertion in \cite[Theorem 2.15]{Green-hypersurfaces}, while Propositions \ref{prop:duality-K+L} and \ref{prop:duality-L} explain why and by how much the second statement of \cite[Theorem 2.15]{Green-hypersurfaces} fails in some crucial cases.

For the sake of completeness, we will recall the construction of generalized Jacobian ideals. Since the constructions and first results make sense for (smooth) hypersurfaces on complex manifolds of arbitrary dimension, we start with this more general setting. Only at the end of the section we will restrict ourselves to the case of curves moving on surfaces.

\subsection{The bundle $\Sigma_L$}

Let $X$ be a smooth compact complex variety of dimension $n$ and $L$ a line bundle on $X$. In this section, we summarize the properties of $\Sigma_L$, the bundle of first-order differential operators on sections of $L$, which is necessary to define the generalized Jacobian ideals and rings. For a more detailed explanation of $\Sigma_L$ with proofs, we refer to \cite{Sernesi-Book}. Note that many of the constructions and first properties are of local nature, hence are also valid for non-compact varieties. However the statements regarding global sections or higher cohomology groups are only meaningful in the compact case.

Suppose $\left(x_1,\ldots,x_n\right)$ are coordinates on an open subset $U\subseteq X$, and that $L_{\mid U}$ is trivialized by a nowhere-vanishing section $e\in H^0\left(U,L\right)$, so that any other section $\sigma\in H^0\left(U,L\right)$ has the form $\sigma=fe$ for a holomorphic function $f\in\caO_X\left(U\right)$.

A local section $D$ of $\Sigma_L$ is then of the form $D=a_0\cdot 1+\sum_{i=1}^n a_i D_i$, where $a_0,\ldots,a_n\in\caO_X\left(U\right)$, and $D$ acts on $\sigma=fe$ by differentiating the function $f$, i.e.
\begin{equation} \label{eq:action-Sigma_L}
    D\left(\sigma\right)=D\left(fe\right):=\left(a_0\cdot f+\sum_{i=1}^na_i\frac{\partial f}{\partial x_i}\right)e.
\end{equation}
Imposing that this formula is compatible with the transition functions for $L$ on the overlap of two such open subsets gives the transition functions for $\Sigma_L$.

The bundle $\Sigma_L$ fits into a natural short exact sequence
\begin{equation} \label{eq:extension-Sigma_L}
    0\longra \caO_X\longra\Sigma_L\longra T_X\longra 0,
\end{equation}
where the inclusion $\caO_X\hookrightarrow\Sigma_L$ is locally given by $a\mapsto a\cdot 1$ (i.e. by considering holomorphic functions as differential operators of order $0$ via multiplication), and the projection $\Sigma_L\to T_X$ is locally given by $D_i\mapsto\frac{\partial}{\partial x_i}$.

\begin{remark} \label{rmk:extension-class-Sigma_L}
The extension class of \eqref{eq:extension-Sigma_L} coincides (up to non-zero scalar multiple) with the first Chern class of $L$
\begin{equation} \label{eq:extension-class-Sigma_L}
    c_1\left(L\right)\in H^1\left(X,\Omega_X^1\right)\cong\Ext_{\caO_X}^1\left(\caO_X,\Omega_X^1\right)\cong \Ext_{\caO_X}^1\left(T_X,\caO_X\right).
\end{equation}
\end{remark}

\subsection{The differential $d\sigma$ of a section $\sigma\in H^0\left(L\right)$ and associated bundles}

\begin{definition} \label{df:differential-section}
Let $\sigma\in H^0\left(X,L\right)$ be a global section. The \emph{differential of $\sigma$} is the morphism of vector bundles
\begin{equation} \label{eq:df:differential-section}
d\sigma\colon\Sigma_L\to L,\quad D\mapsto d\sigma\left(D\right):=D\left(\sigma\right),
\end{equation}
mapping a differential operator $D\in H^0\left(U,\Sigma_L\right)$, defined on an open subset $U\subseteq X$, to the derivative of the restriction $\sigma_{\mid U}$ along $D$.

The kernel of $d\sigma$ is denoted by
\begin{equation} \label{eq:E_sigma}
    \caE_{\sigma}=\ker d\sigma\subseteq\Sigma_L,
\end{equation}
\end{definition}

The following result summarizes the first properties of $d\sigma$ and $\caE_{\sigma}$.

\begin{proposition} \label{prop:differential-section-kernel}
Let $\sigma\in H^0\left(X,L\right)$ be a section of $L$, and denote by $Z=\left\{\sigma=0\right\}\subseteq X$ the vanishing subvariety of $\sigma$. Then:
\begin{enumerate}
    \item The cokernel of $d\sigma\colon\Sigma_L\to L$ is supported on the singular locus of $Z$. In particular, $d\sigma$ is surjective if $Z$ is smooth.
    \item If $Z$ is smooth, then there is a commutative diagram with exact rows and columns
    \begin{equation} \label{diag:generalized-differential-complete}
        \xymatrix{ & & 0 \ar[d] & 0 \ar[d] & \\
         & & \caE_{\sigma} \ar[d] \ar[r]^{\cong} & T_X\left(-\log Z\right) \ar[d] & \\
        0 \ar[r] & \caO_X \ar@{=}[d] \ar[r] & \Sigma_L \ar[r] \ar[d]^{d\sigma} & T_X \ar[r] \ar[d]^{\widetilde{d\sigma}} & 0  \\
        0 \ar[r] & \caO_X  \ar[r]^{\sigma} & L \ar[r] \ar[d] & L_{\mid Z} \ar[r] \ar[d] & 0 \\
         & & 0 & 0}
    \end{equation}
    where $\widetilde{d\sigma}\colon T_X\to L_{\mid Z}$ coincides with the composition $T_X\twoheadrightarrow T_{X\mid Z}\twoheadrightarrow N_{Z/X}\cong L_{\mid Z}$.
    \item If $Z$ is smooth, then
    \begin{equation} \label{eq:det-Esigma}
        \det\caE_{\sigma}\cong K_X^{\vee}\otimes L^{\vee}.
    \end{equation}
\end{enumerate}
\end{proposition}
\begin{proof}
\begin{enumerate}
    \item Follows directly from the local description \eqref{eq:action-Sigma_L} of $d\sigma\left(D\right)=D\left(\sigma\right)$ on a trivialization of $L$, since $Z$ is locally defined by $\left\{f=0\right\}$.

    \item Note first that $d\sigma$ acts on the subsheaf $\caO_X\subseteq\Sigma_L$ by multiplying by $\sigma$. There is thus a commutative diagram
    \begin{equation} \label{diag:generalized-differential}
        \xymatrix{0 \ar[r] & \caO_X \ar@{=}[d] \ar[r] & \Sigma_L \ar[r] \ar[d]^{d\sigma} & T_X \ar[r] \ar@{-->}[d]^{\widetilde{d\sigma}} & 0  \\
        0 \ar[r] & \caO_X  \ar[r]^{\sigma} & L \ar[r] & L_{\mid Z} \ar[r] & 0}
    \end{equation}
    where $\widetilde{d\sigma}\colon T_X\to L_{\mid Z}$ is locally given by differentiating the local equation of $\sigma$ along a local vector field $\chi$ \emph{and restricting it} to $Z$, which turns out to be well-defined, independent of the choice of a lifting of $\chi$ to $\Sigma_L$. Writing the isomorphism $L_{\mid Z}\cong N_{Z/X}$ with the normal bundle in local coordinates gives the alternative description of $\tilde{d\sigma}$ as the composition $T_X\twoheadrightarrow T_{X\mid Z}\twoheadrightarrow N_{Z/X}\cong L_{\mid Z}$.

    With this description, it is clear that a vector field lies in $\ker\widetilde{d\sigma}$ if and only if it is tangent to $Z$ along $Z$. That is, $\ker\widetilde{d\sigma}=T_X\left(-\log Z\right)\subseteq T_X$, the dual of the vector bundle $\Omega_X^1\left(\log Z\right)$ of meromorphic $1$-forms with at most a logarithmic pole along $Z$.

    The diagram \eqref{diag:generalized-differential-complete} is obtained by completing \eqref{diag:generalized-differential} using the snake lemma and the surjectivity of $d\sigma$.

    \item Since both the central row and the central column are short exact sequences of locally free sheaves, we obtain
    \begin{equation} \label{eq:proof:det-Esigma}
        \det\caE_{\sigma}\cong\left(\det\Sigma_L\right)\otimes L^{\vee}\cong\left(\det T_X\right)\otimes L^{\vee}\cong K_X^{\vee}\otimes L^{\vee}.
    \end{equation}
\end{enumerate}
\end{proof}

Using $d\sigma$ we define the following cohomology class $\delta\sigma\in H^1\left(X,\caE_{\sigma}\otimes L^{\vee}\right)$, which will be very useful for certain formalizations.

\begin{definition} \label{df:delta-sigma}
Let $\sigma\in H^0\left(X,L\right)$ with smooth vanishing locus. We denote by
\begin{equation} \label{eq:delta-sigma}
    \delta\sigma:=\left[0\to\caE_{\sigma}\to\Sigma_L\stackrel{d\sigma}{\longra}L\to 0\right]
    \in\Ext_{\caO_X}^1\left(L,\caE_{\sigma}\right)=H^1\left(X,\caE_{\sigma}\otimes L^{\vee}\right),
\end{equation}
the extension class of the central column of \eqref{diag:generalized-differential}.
\end{definition}

Note that $\delta\sigma$ coincides with the image of $1\in H^0\left(X,\caO_X\right)$, under the connecting homomorphism of the same extension tensored by $L^{\vee}$:
\begin{equation} \label{eq:extension-dsigma}
    0 \longra \caE_{\sigma}\otimes L^{\vee}\longra \Sigma_L\otimes L^{\vee}\stackrel{d\sigma}{\longra} \caO_X \longra 0.
\end{equation}
The Koszul komplex of $d\sigma\colon\Sigma_L\otimes L^{\vee}\to\caO_X$ has the form
\begin{equation} \label{eq:Koszul-dsigma}
    0\longra\bigwedge^{n+1}\Sigma_L\otimes L^{-n-1}\longra\bigwedge^n\Sigma_L\otimes L^{-n}\longra\cdots \longra\bigwedge^2\Sigma_L\otimes L^{-2}\longra\Sigma_L\otimes L^{\vee}\longra\caO_X\longra 0
\end{equation}
Using the isomorphism $\det\Sigma_L\cong\det\caE_{\sigma}\otimes L$ at the first term, the Koszul complex can be split into short exact sequences
\begin{equation} \label{eq:ses-Koszul-dsigma}
    0\longra\bigwedge^r\caE_{\sigma}\otimes L^{-r}\longra\bigwedge^r\Sigma_L\otimes L^{-r}\longra\bigwedge^{r-1}\caE_{\sigma}\otimes L^{-r+1}\longra 0,
\end{equation}
for $r=1,\ldots,n$, whose extension classes all coincide with $\delta\sigma$, under the canonical identification $\Ext_{\caO_X}^1\left(\bigwedge^{r-1}\caE_{\sigma}\otimes L^{-r+1},\bigwedge^r\caE_{\sigma}\otimes L^{-r}\right)\cong\Ext_{\caO_X}^1\left(L,\caE_{\sigma}\right).$

\subsection{Generalized Jacobian rings}

Let now $M$ be another line bundle on $X$. We still denote with $d\sigma$ the map $\Sigma_L\otimes L^{\vee}\otimes M\to M$ induced by tensoring $d\sigma$ with the identity on $M$. Also, in order to lighten the notation, we will remove the variety $X$ from the notation of cohomology groups from now on, simply writing $H^i\left(\caF\right)$ instead $H^i\left(X,\caF\right)$ for any sheaf $\caF$ on $X$.

\begin{definition} \label{df:Jacobian-idea-ring}
Due to Green \cite{Green-hypersurfaces}, we define
\begin{enumerate}
\item the \emph{generalized Jacobian ideal} (with respect to $\sigma$)
\begin{equation} \label{eq:generalized-Jacobian}
J_M:=\im\left(d\sigma\colon H^0\left(\Sigma_L\otimes M\otimes L^{\vee}\right)\longra H^0\left(M\right)\right)\subseteq H^0\left(M\right)
\end{equation}
\item and the \emph{generalized Jacobian ring} $R_M:= H^0\left(M\right)/J_M$
\end{enumerate}
\emph{on degree $M$}.
\end{definition}

Note that $H^0\left(M\right)$ is not a ring, so neither $J_M\subseteq H^0\left(M\right)$ is an ideal nor $R_M$ is a ring. All these spaces are merely vector spaces. However, for any $M,N\in\Pic\left(X\right)$, the natural multiplication maps $H^0\left(M\right)\otimes H^0\left(N\right)\to H^0\left(M\otimes N\right)$ map the subspaces $J_M\otimes H^0\left(N\right)$ and $H^0\left(M\right)\otimes J_N$ into $J_{M\otimes N}$, and this also allows us to define a multiplication map $R_M\otimes R_N\to R_{M\otimes N}$. Loosely speaking, although the following construction only works if $h^1\left(\caO_X\right)=0$, we can think of the direct sum $\bigoplus_M H^0\left(M\right)$ (over all isomorphism classes of line bundles $M$) as a commutative ring graded by $\Pic\left(X\right)$, and then the subspace $\bigoplus_M J_M$ forms a graded ideal with quotient ring $\bigoplus_M R_M$. The generalized Jacobian ideal $J_M$ and ring $R_M$ are just the summands of degree $M$.

\begin{example}
For $S=\mP^2$ (with homogeneous coordinates $\left[x_0\colon x_1\colon x_2\right]$) and $L=\caO_{\mP^2}\left(n\right)$ with $n\neq 0$, we have $\Sigma_L\cong\caO_{\mP^2}\left(1\right)^{\oplus 3}$, and the sequence \eqref{eq:extension-Sigma_L} is the Euler sequence, with the first map appropriately rescaled.
$$\xymatrix@R-2pc{
0 \ar[r] & \caO_{\mP^2} \ar[r] & \caO_{\mP^2}\left(1\right)^{\oplus 3} \ar[r] & T_{\mP^2} \ar[r] & 0 \\
 & 1 \ar@{|->}[r] & \frac{1}{n}\left(x_0,x_1,x_2\right) & &
}$$
For any $F\in H^0\left(\mP^2,L\right)=\mC\left[x_0,x_1,x_2\right]_n$, the differential $dF\colon\caO_{\mP^2}\left(1\right)^{\oplus 3}\to \caO_{\mP^2}\left(n\right)$ is given by the partial derivatives $\left(\frac{\partial F}{\partial x_0},\frac{\partial F}{\partial x_1},\frac{\partial F}{\partial x_2}\right)$. The rescaling by $\frac{1}{n}$ ensures that $dF$, restricted to $\caO_{\mP^2}$, coincides with the $F$, since the Euler relation gives
$$dF\left(\frac{1}{n}\left(x_0,x_1,x_2\right)\right)=\frac{1}{n}\left(x_0\frac{\partial F}{\partial x_0}+x_1\frac{\partial F}{\partial x_1}+x_2\frac{\partial F}{\partial x_2}\right)=F.$$

The generalized Jacobian ideal in ``degree'' $\caO_{\mP^2}\left(m\right)$ is the image of
$$H^0\left(\caO_{\mP^2}\left(m-n+1\right)\right)^{\oplus 3}\stackrel{dF}{\longra} H^0\left(\caO_{\mP^2}\left(m\right)\right),$$
i.e. it is generated by the partial derivatives of $F$ multiplied by arbitrary polynomials of the correct degree. This means, $J_{\caO_{\mP^2}\left(m\right)}=\left(\frac{\partial F}{\partial x_0},\frac{\partial F}{\partial x_1},\frac{\partial F}{\partial x_2}\right)_m$ is precisely the piece of degree $m$ of the usual Jacobian ideal $J^F\subseteq\mC\left[x_0,x_1,x_2\right]$, which motivates the name.
\end{example}

Since for any $M$ and any $r=1,\ldots,n$ the extension class of
\begin{equation} \label{eq:split-Koszul}
    0\longra\bigwedge^r\caE_{\sigma}\otimes L^{-r}\otimes M\longra\bigwedge^r\Sigma_L\otimes L^{-r} \stackrel{d\sigma}{\longra} \bigwedge^{r-1}\caE_{\sigma}\otimes L^{-r+1}\otimes M\longra 0
\end{equation}
is still $\delta\sigma\in H^1\left(\caE_{\sigma}\otimes L^{\vee}\right)$, the coboundary morphisms \begin{equation} \label{eq:conn-homom-delta-sigma}
    H^{r-1}\left(\bigwedge^{r-1}\caE_{\sigma}\otimes L^{-r+1}\otimes M\right)\longra H^r\left(\bigwedge^r\caE_{\sigma}\otimes L^{-r}\otimes M\right)
\end{equation}
are given by cup-product with $\delta\sigma$. In particular, we have
\begin{equation} \label{eq:J_M-ker}
J_M=\im\left(d\sigma\colon H^0\left(\Sigma_L\otimes M\otimes L^{\vee}\right)\to H^0\left(M\right)\right)=\ker\left(\delta\sigma\colon H^0\left(M\right)\to H^1\left(\caE_{\sigma}\otimes L^{\vee}\otimes M\right)\right),
\end{equation}
and thus, cup-product with $\delta\sigma$ induces an inclusion
\begin{multline} \label{eq:R_M-in-H1}
    R_M=\coker\left(d\sigma\colon H^0\left(\Sigma_L\otimes M\otimes L^{\vee}\right)\to H^0\left(M\right)\right)\stackrel{\delta\sigma}{\cong}\\
    \cong\ker\left(H^1\left(\caE_{\sigma}\otimes L^{\vee}\otimes M\right)\to H^1\left(\Sigma_L\otimes L^{\vee}\otimes M\right)\right)\subseteq H^1\left(\caE_{\sigma}\otimes L^{\vee}\otimes M\right).
\end{multline}
Note that the last inclusion can be an equality, for example, if $H^1\left(\Sigma_L\otimes L^{\vee}\otimes M\right)=0$.

\subsection{Jacobian in degree $K_X\otimes L$}

We now focus on the generalized Jacobian ideal and ring on degree $K_X\otimes L$, which turns out to be the piece on which the cup-product with the Kodaira-Spencer class of a general embedded deformation should become an isomorphism.

\begin{proposition} \label{prop:R-K-L-dim-n}
Suppose $L$ is ample and $\sigma\in H^0\left(L\right)$ is a global section, with smooth vanishing divisor $Z:=\left\{\sigma=0\right\}$. Then, the following hold:
\begin{enumerate}
\item $J_{K_X\otimes L}=H^0\left(K_X\right)\cdot\sigma=\im\left(H^0\left(K_X\right)\stackrel{\cdot\sigma}{\to}H^0\left(K_X\otimes L\right)\right)$. \label{item:lem:R-K-L-dim-n:1}
\item The adjunction formula induces a natural inclusion $R_{K_X\otimes L}\hookrightarrow H^0\left(K_Z\right)$. \label{item:lem:R-K-L-dim-n:2}
\item The natural restriction morphism
\begin{equation}
H^0\left(\Omega_X^{n-1}\right)\longra H^0\left(\Omega_{X\mid Z}^{n-1}\right)\longra H^0\left(K_Z\right),
\end{equation}
is injective, hence we can identify $H^0\left(\Omega_X^{n-1}\right)$ with its image in $H^0\left(K_Z\right)$. \label{item:lem:R-K-L-dim-n:3}
\item Moreover, it holds $H^0\left(K_Z\right)=H^0\left(\Omega_X^{n-1}\right)\oplus R_{K_X\otimes L}$, and in particular, $\dim R_{K_X\otimes L}=p_g\left(Z\right)-h^0\left(\Omega_X^{n-1}\right)$. \label{item:lem:R-K-L-dim-n:4}
\end{enumerate}
\end{proposition}
\begin{proof}
Twisting the diagram \eqref{diag:generalized-differential} with $K_X$ and using the adjunction formula for $Z\subseteq X$ we obtain
\begin{equation} \label{diag:proof:lem-R-K-L-0}
    \xymatrix{0 \ar[r] & K_X \ar[r] \ar@{=}[d] & \Sigma_L\otimes K_X \ar[r] \ar[d]^{d\sigma} & \Omega_X^{n-1} \ar[r] \ar[d] & 0  \\
    0 \ar[r] & K_X \ar[r]^{\sigma} & K_X\otimes L \ar[r] & \left(K_X\otimes L\right)_{\mid Z}\cong K_Z \ar[r] & 0}
\end{equation}
where the right vertical arrow is the restriction of $\left(n-1\right)$-forms from $X$ to $Z$.

\begin{enumerate}
\item The connecting homomorphism $H^0\left(\Omega_X^{n-1}\right)\to H^1\left(K_X\right)$ of the long exact sequence in cohomology of the upper row is given by cup product with the corresponding extension class, which is $c_1\left(L\right)\in H^1\left(\Omega_X^1\right)$. Since $L$ is ample, Lefschetz theorem implies that the connecting homomorphism is an isomorphism. Thus there is an equality $H^0\left(K_X\right)=H^0\left(\Sigma_L\otimes K_X\right)$ and
\begin{equation} \label{eq:proof:lem-R-K-L-1}
J_{K_X\otimes L}=\im\left(d\sigma\colon H^0\left(\Sigma_L\otimes K_X\right)\longra H^0\left(K_X\otimes L\right)\right)=\im\left(\cdot\sigma\colon H^0\left(K_X\right)\longra H^0\left(K_X\otimes L\right)\right).
\end{equation}

\item Assuming (\ref{item:lem:R-K-L-dim-n:1}) we have
\begin{multline} \label{eq:proof:lem-R-K-L-2}
R_{K_X\otimes L}=H^0\left(K_X\otimes L\right)/J_{K_X\otimes L}=\coker\left(H^0\left(K_X\right)\to H^0\left(K_X\otimes L\right)\right)=\\
=\ker\left(H^0\left(K_Z\right)\longra H^1\left(K_X\right)\right)\subseteq H^0\left(K_Z\right).
\end{multline}

\item The first morphism $H^0\left(\Omega_X^{n-1}\right)\to H^0 \left(\Omega_{X\mid Z}^{n-1}\right)$ is injective, since its kernel is
\begin{equation} \label{eq:proof:lem-R-K-L-3-1}
H^0\left(\Omega_X^{n-1}\left(-Z\right)\right)\cong H^n\left(\Omega_X^1\left(Z\right)\right)^*=0,
\end{equation}
where the first isomorphism follows from Serre duality, and the second equality follows from Kodaira-Nakano vanishing theorem.

The second morphism $H^0\left(\Omega_{X\mid Z}^{n-1}\right)\ra H^0\left(K_Z\right)$ is also injective, since its kernel is
\begin{equation} \label{eq:proof:lem-R-K-L-3-2}
H^0\left(T_Z\otimes K_{X\mid Z}\right)=H^0\left(T_Z\otimes K_Z\otimes L_{\mid Z}^{\vee}\right)\cong H^{n-1}\left(\Omega_Z^1\otimes L_{\mid Z}\right)^*,
\end{equation}
which vanishes again by Kodaira-Nakano (on the hypersurface $Z$), because $L_{\mid Z}$ is ample.

\item From diagram \eqref{diag:proof:lem-R-K-L-0}, we obtain a commutative diagram
\begin{equation} \label{diag:proof-R-K-L-4}
\xymatrix{H^0\left(\Omega_X^{n-1}\right) \ar[r]^{\cdot\left(c_1\left(L\right)\right)}_{\cong} \ar@{^(->}[d] & H^1\left(K_X\right) \ar@{=}[d] \\
H^0\left(K_Z\right) \ar[r]^{\nu} & H^1\left(K_X\right)
}
\end{equation}
Since the upper map is an isomorphism, we recover the injectivity of the left map and also obtain that $H^0\left(\Omega_X^{n-1}\right)\oplus\ker\nu=H^0\left(K_Z\right)$. But $\ker\nu=R_{K_X\otimes L}$ by the proof of (\ref{item:lem:R-K-L-dim-n:2}), and the proof is finished.
\end{enumerate}
\end{proof}

\begin{example}
    For $X=\mP^2$, we have $H^0\left(K_X\right)=H^0\left(\Omega_X^1\right)=0$. Hence, for any $Z\subseteq X$ defined by an homogeneous polynomial of degree $n$, Proposition \ref{prop:R-K-L-dim-n} recovers $J_{K_X\otimes L}=J_{n-3}=0$ (since the Jacobian ideal is generated by polynomials of degree $n-1$) and $H^0\left(\Omega_Z^1\right)\cong R_{n-3}=\mC\left[x_0,x_1,x_2\right]_{n-3}$.
\end{example}

\subsection{Jacobian in degree $L$} \label{subsect:degree-L}

We have seen in Proposition \ref{prop:R-K-L-dim-n} that the Jacobian ring $R_{K_X\otimes L}$ of degree $K_X\otimes L$ is a complement in $H^0\left(K_Z\right)$ of the set of top-forms on $Z$, obtained by restricting global $\left(n-1\right)$-forms of $X$. This makes it a good candidate to be the ``most variable'' part of $H^0\left(K_Z\right)$, under embedded deformations of $Z$ in $X$. We now check that the Jacobian ring $R_L$ of degree $L=\caO_X\left(Z\right)$ is a good tool to encode the (Kodaira-Spencer classes of) first-order deformations of $Z$ within the linear system $\left|L\right|$.

The identification $\left|L\right|=\mP\left(H^0\left(L\right)\right)$ also gives $T_{\left[Z\right]}\left|L\right|\cong H^0\left(L\right)/\left\langle\sigma\right\rangle$, where $Z=\left\{\sigma=0\right\}$. This means that every first order deformation of $Z$ inside the linear system $\left|L\right|$ is given by $\left\{\sigma+\epsilon\tau=0\right\}$, where $\tau\in H^0\left(L\right)$ is another section and $\epsilon$ is the parameter (with $\epsilon^2=0$).

We now consider the following two short exact sequences of sheaves:
\begin{equation} \label{eq:restriction}
0 \longra \caO_X \stackrel{\sigma}{\longra} L \longra L_{\mid Z}\cong N_{Z/X}\longra 0
\end{equation}
and
\begin{equation} \label{eq:normal-sequence}
0 \longra T_Z \longra T_{X\mid Z}\longra N_{Z/X}\longra 0.
\end{equation}

From \eqref{eq:restriction}, we obtain an inclusion $T_{\left[Z\right]}\left|L\right|=H^0\left(L\right)/\left\langle\sigma\right\rangle\hookrightarrow H^0\left(N_{Z/X}\right)$, which composed with the first connecting homomorphism $H^0\left(N_{Z/X}\right)\to H^1\left(T_Z\right)$ of \eqref{eq:normal-sequence} gives a map
\begin{equation} \label{eq:restriction-KS}
KS\colon T_{\left[Z\right]}\left|L\right|\longra H^1\left(T_Z\right).
\end{equation}
This is precisely the Kodaira-Spencer map for the family of deformations of $Z$, inside the linear system $\left|L\right|$.

\begin{lemma} \label{lem:KS-RL}
Let $X$ be a smooth compact complex manifold, $Z\subseteq X$ a smooth hypersurface and $L=\caO_X\left(Z\right)$ the associated line bundle. Then the Kodaira-Spencer map \eqref{eq:restriction-KS} factorizes as
\begin{equation} \label{eq:KS-RL}
    T_{\left[Z\right]}\left|L\right|=H^0\left(L\right)/\left\langle\sigma\right\rangle\longra R_L=H^0\left(L\right)/J_L\longra H^1\left(T_Z\right)
\end{equation}
\end{lemma}
\begin{proof}
Consider the commutative diagram
\begin{equation} \label{diag:proof-R-L-KS-1}
    \xymatrix{0 \ar[r] & \caO_X \ar[r] \ar@{=}[d] & \Sigma_L \ar[r] \ar[d]^{d\sigma} & T_X \ar[r] \ar[d]^{\widetilde{d\sigma}} & 0  \\
    0 \ar[r] & \caO_X \ar[r]^{\cdot\sigma} & L \ar[r] & L_{\mid Z}\cong N_{Z/X} \ar[r] & 0}
\end{equation}
Note first that $\sigma=d\sigma\left(1\right)\in J_L$, hence the first projection $H^0\left(L\right)/\left\langle\sigma\right\rangle\to R_L$ in \eqref{eq:KS-RL} is well-defined and has kernel $J_L/\left\langle\sigma\right\rangle$.

It remains to show that $J_L/\left\langle\sigma\right\rangle\subseteq \ker\left(KS\right)$. To this aim, recall that the map $T_X\to N_{Z/X}$ factorizes as $T_X\to T_{X\mid Z}\to N_{Z/X}$ (Proposition \ref{prop:differential-section-kernel}), and thus
\begin{multline} \label{eq:proof-R-L-KS-2}
J_L/\left\langle\sigma\right\rangle=\im\left(H^0\left(\Sigma_L\right)\to H^0\left(N_{Z/X}\right)\right)\subseteq \im\left(H^0\left(T_X\right)\to H^0\left(N  _{Z/X}\right)\right)\subseteq \\
\subseteq\im\left(H^0\left(T_{X\mid Z}\right)\to H^0\left(N_{Z/X}\right)\right)=\ker\left(H^0\left(N_{Z/X}\right)\stackrel{KS}{\longra}H^1\left(T_Z\right)\right).
\end{multline}
\end{proof}

\begin{definition} \label{df:KS-RL}
By abuse of notation, we will write $KS\colon H^0\left(L\right)\to H^1\left(T_Z\right)$ and $KS\colon R_L\to H^1\left(T_Z\right)$, for the maps induced by \eqref{eq:restriction-KS}, and we will call all three maps indistinctly \emph{Kodaira-Spencer maps}.
\end{definition}

\subsection{The top degree}

We now consider a third (and last) interesting degree, that plays the role of the socle degree in the case of $\mP^2$.

We consider the composition of $n$ connecting homomorphisms, as in \eqref{eq:conn-homom-delta-sigma} (with suitably varying $M$'s)
$$H^0\left(M\right)\longra H^n\left(\bigwedge^n\caE_{\sigma}\otimes L^{-n}\otimes M\right)=H^n\left(\det\caE_{\sigma}\otimes L^{-n}\otimes M\right),$$
which is given by cup-product with the cup-power $\left(\delta\sigma\right)^n\in H^n\left(\bigwedge^n\caE_{\sigma}\otimes L^{-n}\right)$. Composing with the natural isomorphism
\begin{equation} \label{eq:isom-determinants}
    \det\caE_{\sigma}\otimes L^{-n}\cong\det\Sigma_L\otimes L^{-n-1}\cong K_X^{\vee}\otimes L^{-n-1},
\end{equation}
we then obtain a map
\begin{equation} \label{eq:mu-sigma-1}
H^0\left(M\right)\stackrel{\left(\delta\sigma\right)^n}{\longra} H^n\left(K_X^{\vee}\otimes L^{-n-1}\otimes M\right).
\end{equation}
Since the kernel of the first product with $\delta\sigma$ is $J_M$, the above map factors through the generalized Jacobian ring as
\begin{equation} \label{eq:mu-sigma-2}
R_M\to H^n\left(K_X^{\vee}\otimes L^{-n-1}\otimes M\right).
\end{equation}
In particular, for $M=K_X^2\otimes L^{n+1}$ we obtain a map
\begin{equation} \label{eq:natural-trace-map}
    \tr=\tr_{\sigma}\colon R_{K^2\otimes L^{n+1}}\longra H^n\left(K_X\right)\cong\mC.
\end{equation}

\begin{definition} \label{df:natural-trace-map}
    We call the map $\tr$ in \eqref{eq:natural-trace-map} \emph{natural trace map} with respect to the given section $\sigma\in H^0\left(L\right)$.
\end{definition}

We now investigate under which conditions the map \eqref{eq:natural-trace-map} is an isomorphism.

\begin{lemma} \label{lem:top-isomorphism}
If $L$ is sufficiently ample, more precisely if $H^1\left(\Sigma_L\otimes K_X^2\otimes L^n\right)=0$ and
\begin{equation} \label{eq:condition-top-isomorphism}
    H^{r-1}\left(\bigwedge^r\Sigma_L\otimes K_X^2\otimes L^{n+1-r}\right)=H^r\left(\bigwedge^r\Sigma_L\otimes K_X^2\otimes L^{n+1-r}\right)=0
\end{equation}
for all $r=2,\ldots,n$, then $\tr_{\sigma}\colon R_{K_X^2\otimes L^{n+1}}\to H^n\left(K_X\right)\cong\mC$ is an isomorphism.
\end{lemma}
\begin{proof}
Set $M=K_X^2\otimes L^{n+1}$ in the exact sequences \eqref{eq:split-Koszul}, which become
\begin{equation} \label{eq:split-Koszul-top-degree}
0\longra\bigwedge^r\caE_{\sigma}\otimes K_X^2\otimes L^{n+1-r}\longra\bigwedge^r\Sigma_L\otimes K_X^2\otimes L^{n+1-r}\longra\bigwedge^{r-1}\caE_{\sigma}\otimes K_X^2\otimes L^{n+2-r}\longra 0.
\end{equation}
The condition $H^1\left(\Sigma_L\otimes K_X^2\otimes L^n\right)=0$ ensures that $R_{K_X^2\otimes L^{n+1}}\stackrel{\delta\sigma}{\cong} H^1\left(\caE_{\sigma}\otimes K_X^2\otimes L^n\right)$. The remaining conditions ensure that the connecting homomorphisms
$$H^{r-1}\left(\bigwedge^{r-1}\caE_{\sigma}\otimes K_X^2\otimes L^{n+2-r}\right)\stackrel{\delta\sigma}{\longra} H^r\left(\bigwedge^r\caE_{\sigma}\otimes K_X^2\otimes L^{n+1-r}\right)$$
are isomorphisms. Thus, multiplication with $\left(\delta\sigma\right)^n$ followed by the isomorphism \eqref{eq:isom-determinants} gives indeed an isomorphism
$$R_{K_X^2\otimes L^{n+1}}\cong H^n\left(\det\caE_{\sigma}\otimes K_X^2\otimes L\right)\cong H^n\left(K_X\right)\cong\mC.$$
\end{proof}

\begin{remark} \label{rmk:top-isomorphism-sufficiently-ample}
    The vanishing conditions in Lemma \ref{lem:top-isomorphism} not only involve $L$ but also $\Sigma_L$. We haven't deduced any ampleness of $\Sigma_L$ from the ampleness of $L$, so it might happen that the exterior powers of $\Sigma_L$ somehow cancel the ampleness of $L$ and that the conditions \eqref{eq:condition-top-isomorphism} are rarely satisfied. This is, however, not the case at all, because of the first fundamental exact sequence \eqref{eq:extension-Sigma_L}, which gives also exact sequences
    $$0 \longra \bigwedge^{r-1}T_X \longra \bigwedge^r\Sigma_L \longra \bigwedge^r T_X \longra 0$$
    and thus, tensoring with $K_X^2\otimes L^{n+1-r}$, we also obtain
    $$0 \longra \Omega_X^{n-r+1}\otimes K_X\otimes L^{n+1-r} \longra \bigwedge^r\Sigma_L\otimes K_X^2\otimes L^{n+1-r} \longra \Omega_X^{n-r}\otimes K_X\otimes L^{n+1-r} \longra 0.$$
    The outer terms of the above exact sequence have no higher cohomology, if $L$ is sufficiently ample, and thus the vanishing conditions of Lemma \ref{lem:top-isomorphism} are satisfied.
\end{remark}

\begin{example}
For $X=\mP^2$ and $Z\subseteq\mP^2$ defined by a polynomial of degree $n$, Lemma \ref{lem:top-isomorphism} is precisely the part of and Macaulay's theorem stating that the (one-dimensional) socle of the Jacobian ring is the piece of degree $3n-6$. Indeed, in this case we have $K_X^2\otimes L^3=\caO_{\mP^2}\left(3n-6\right)$.
\end{example}

\subsection{Natural pairings and duality maps}

For any line bundle $M$, let $M':=K_X^2\otimes L^{n+1}\otimes M^{\vee}$ denote the line bundle such that $M\otimes M'\cong K_X^2\otimes L^{n+1}$. Combining the multiplication map with the natural trace map $\tr_{\sigma}\colon R_{K_X^2\otimes L^{n+1}}\to H^n\left(K_X\right)\cong\mC$, we obtain a pairing
\begin{equation} \label{eq:pairing}
R_M\otimes R_{M'}\longra R_{K_X^2\otimes L^{n+1}}\longra\mC,
\end{equation}
and thus, also (mutually dual) duality maps
\begin{equation} \label{eq:duality}
\lambda_M\colon R_M\longra R_{M'}^* \quad\text{and}\quad \lambda_{M'}\colon R_{M'}\longra R_M^*.
\end{equation}

\begin{remark}
    The main content of Macaulay's theorem for curves in $\mP^2$ (or hypersurfaces in $\mP^n$) is precisely that the pairings \eqref{eq:pairing} are perfect.
\end{remark}

Of course, if the natural trace is not an isomorphism, there is little hope for the pairing \eqref{eq:pairing} to be perfect. But even if $R_{K_X^2\otimes L^{n+1}}\cong\mC$ holds, the pairing will not be perfect in general. Thus, we cannot expect the duality maps $\lambda_M$ to be isomorphisms, but only aim for them to have maximal rank.

We now study the particular cases of duality maps that will be of most interest to us, namely when $X$ is a surface and $M=L$ or $M=K_X\otimes L$.

\begin{proposition} \label{prop:duality-K+L}
If $X$ is a surface and $L$ is an ample line bundle such that $H^1\left(\Omega_X^1\otimes L\right)=0$, then the duality map $\lambda_{K_X\otimes L^2}$ fits into a natural exact sequence
\begin{equation} \label{eq:duality-K+L}
0 \longra H^0\left(\caO_X\right)\stackrel{\cdot c_1\left(L\right)}{\longra} H^1\left(\Omega_X^1\right)\longra R_{K_X\otimes L^2} \stackrel{\lambda_{K_X\otimes L^2}}{\longra} R_{K_X\otimes L}^*\longra 0.
\end{equation}
\end{proposition}

\begin{proof}
Recall from \eqref{eq:R_M-in-H1} that $\delta\sigma$ induces the isomorphism
\begin{equation} \label{eq:duality-K+L-1}
    R_{K_X\otimes L^2} \stackrel{\cdot\delta\sigma}{\cong}\ker\left(H^1\left(\caE_{\sigma}\otimes K_X\otimes L\right)\longra H^1\left(\Sigma_L\otimes K_X\otimes L\right)\right).
\end{equation}

Twisting \eqref{eq:extension-Sigma_L} with $K_X\otimes L$, we obtain $0\to K_X\otimes L\to \Sigma_L\otimes K_X \otimes L \to \Omega_X^1\otimes L \to 0$. Taking into account the hypothesis $H^1\left(\Omega_X^1\otimes L\right)=0$, and that $H^1\left(K_X\otimes L\right)=0$ by Kodaira vanishing, the long exact sequence in cohomology gives
\begin{equation} \label{eq:duality-K+L-2}
    H^1\left(\Sigma_L\otimes K_K\otimes L\right)=0,\quad\text{and thus}\quad R_{K_X\otimes L^2}\cong H^1\left(\caE_{\sigma}\otimes K_X\otimes L\right).
\end{equation}

Again from \eqref{eq:R_M-in-H1}, we have an isomorphicm
\begin{equation} \label{eq:duality-K+L-3}
    R_{K_X\otimes L} \stackrel{\cdot\delta\sigma}{\cong} \ker\left(H^1\left(\caE_{\sigma}\otimes K_X\right)\longra H^1\left(\Sigma_L\otimes K_X\right)\right).
\end{equation}
Since $H^1\left(K_X\otimes L\right)=0$, because $L$ is ample, we obtain an exact sequence
\begin{equation} \label{eq:duality-K+L-4}
    0\longra R_{K_X\otimes L}\longra H^1\left(\caE_{\sigma}\otimes K_X\right)\longra H^1\left(\Sigma_L\otimes K_X\right)\longra H^1\left(K_X\otimes L\right)=0.
\end{equation}

Consider finally the short exact sequence \eqref{eq:extension-Sigma_L} twisted by $K_X$:
\begin{equation} \label{eq:duality-K+L-5}
    0\longra K_X\longra \Sigma_L\otimes K_X\longra \Omega_X^1\longra 0,
\end{equation}
whose connecting homomorphisms in cohomology are given, up to sign, by cup-product with $c_1\left(L\right)$. In particular, we have
\begin{equation} \label{eq:duality-K+L-6}
    H^0\left(\Omega_X^1\right)\stackrel{c_1\left(L\right)}{\longra} H^1\left(K_X\right)\stackrel{0}{\longra} H^1\left(\Sigma_L\otimes K_X\right)\longra H^1\left(\Omega_X^1\right) \stackrel{c_1\left(L\right)}{\longra} H^2\left(K_X\right) \longra 0,
\end{equation}
where the first cup-product with $c_1\left(L\right)$ is an isomorphism, and the second one is surjective, because of Lefschetz theorem. Thus, in particular, we have a short exact sequence
\begin{equation} \label{eq:duality-K+L-7}
    0 \longra H^1\left(\Sigma_L\otimes K_X\right)\longra H^1\left(\Omega_X^1\right)\stackrel{c_1\left(L\right)}{\longra} H^2\left(K_X\right)\longra 0,
\end{equation}
which together with \eqref{eq:duality-K+L-4} gives an exact sequence
\begin{equation} \label{eq:duality-K+L-8}
    0\longra R_{K_X\otimes L}\longra H^1\left(\caE_{\sigma}\otimes K_K\right)\longra H^1\left(\Omega_X^1\right)\stackrel{c_1\left(L\right)}{\longra} H^2\left(K_X\right) \longra 0.
\end{equation}
The statement follows after dualizing \eqref{eq:duality-K+L-8} and using Serre duality in $H^2\left(K_X\right)^*\cong H^0\left(\caO_X\right)$ and
\begin{equation} \label{eq:duality-K+L-9}
    H^1\left(\caE_{\sigma}\otimes K_X\right)^*\cong H^1\left(\caE_{\sigma}^{\vee}\right)\cong H^1\left(\caE_{\sigma}\otimes K_X\otimes L\right)\stackrel{\eqref{eq:duality-K+L-2}}{\cong} R_{K_X\otimes L^2},
\end{equation}
where the isomorphism in the middle follows from
\begin{equation} \label{eq:duality-K+L-10}
    \caE_{\sigma}^{\vee}\cong\det\caE_{\sigma}^{\vee}\otimes \caE \cong \left(\det\Sigma_L^{\vee}\right)\otimes L\otimes\caE_{\sigma}\cong K_X\otimes L\otimes \caE_{\sigma}.
\end{equation}
\end{proof}

\begin{proposition} \label{prop:duality-L}
Suppose that $X$ is a projective surface. If $K_X\otimes L$ is ample and $H^1\left(\Omega_X^1\otimes K_X\otimes L\right)=0$, the duality map $\lambda_{K_X^2\otimes L^2}$ fits into a natural exact sequence
\begin{equation} \label{eq:duality-L}
H^1\left(L\right)^*\stackrel{H^1\left(d\sigma\right)^*}{\longra} H^1\left(\Sigma_L\right)^*\longra R_{K_X^2\otimes L^2} \stackrel{\lambda_{K_X^2\otimes L^2}}{\longra} R_L^*\longra 0.
\end{equation}
\end{proposition}

\begin{proof}
As in the previous proof, using \eqref{eq:R_M-in-H1}, we obtain
\begin{equation} \label{eq:R-K2-L2-H1}
    R_{K_X^2\otimes L^2} \stackrel{\cdot\delta\sigma}{\cong} \ker\left(H^1\left(\caE_{\sigma}\otimes K_X^2\otimes L\right)\longra H^1\left(\Sigma_L\otimes K_X^2\otimes L\right)\right).
\end{equation}
From \eqref{eq:extension-Sigma_L}, $H^1\left(\Sigma_L\otimes K_X^2\otimes L\right)$ sits into the exact sequence
$$0=H^1\left(K_X^2\otimes L\right)\longra H^1\left(\Sigma_L\otimes K_X^2\otimes L\right)\longra H^1\left(T_X\otimes K_X^2\otimes L\right)\cong H^1\left(\Omega_X^1\otimes K_X\otimes L\right)=0,$$
where the first term vanishes because $K_X\otimes L$ is ample, and the last term vanishes by hypothesis. Hence, also $H^1\left(\Sigma_L\otimes K_X^2\otimes L\right)=0$ and \eqref{eq:R-K2-L2-H1} turns into an isomorphism $R_{K_X^2\otimes L^2}\stackrel{\cdot\delta\sigma}{\cong}H^1\left(\caE_{\sigma}\otimes K_X^2\otimes L\right)$.

By Serre Duality and the natural isomorphism $\caE_{\sigma}^{\vee} \cong \caE_{\sigma}\otimes K_X\otimes L$, we also obtain
\begin{equation} \label{eq:R-K2-L2-dual}
R_{K_X^2\otimes L^2}\cong H^1\left(\caE_{\sigma}^{\vee}\otimes K_X\right)\cong H^1\left(\caE_{\sigma}\right)^*.
\end{equation}

On the other hand, also from \eqref{eq:R_M-in-H1} we can identify $R_L\stackrel{\cdot\delta\sigma}{\cong}\ker\left(H^1\left(\caE_{\sigma}\right)\to H^1\left(\Sigma_L\right)\right)$, hence we have an exact sequence
$$0\longra R_L\stackrel{\cdot\delta\sigma}{\longra} H^1\left(\caE_{\sigma}\right)\longra H^1\left(\Sigma_L\right)\stackrel{H^1\left(d\sigma\right)}{\longra} H^1\left(L\right).$$
Dualizing it and using the natural isomorphism \eqref{eq:R-K2-L2-dual} we obtain the statement.
\end{proof}

\section{IVHS of curves in surfaces with generalized Jacobian rings}

\label{sect:curves-surfaces}

In this section, we show how the generalized Jacobian rings in the degrees studied above encode almost all necessary information about IVHS of an ample smooth curve $Z$ in a surface $X$.

In general, an (one-parameter) IVHS of a smooth curve $Z$ is given by a symmetric linear map
\begin{equation} \label{eq:IVHS-general}
H^0\left(K_Z\right)\to H^1\left(\caO_Z\right)\cong H^0\left(K_Z\right)^*.
\end{equation}
If the IVHS comes from first-order infinitesimal deformation of $Z$, then it is given (up to non-zero scalar) by cup-product with the Kodaira-Spencer class.

Assume from now on that $X$ is a smooth projective surface and $Z\subseteq X$ is the (smooth) vanishing locus of a section $\sigma$ of an ample line bundle $L$. This gives the whole family of generalized Jacobian rings $R_M$ for $M\in\Pic\left(X\right)$, with the properties developed in the previous section.

If we consider an embedded first-order deformation given by another section $\tau\in H^0\left(L\right)$, the associated IVHS is given by cup-product with $\xi_{\tau}:= KS\left(\tau\right)\in H^1\left(T_Z\right)$, where $KS\colon H^0\left(L\right)\to H^1\left(T_Z\right)$ is as defined in Definition \ref{df:KS-RL} (see Lemma \ref{lem:KS-RL}). In particular, $\xi_{\tau}$ depends only on the class of $\tau$ in $R_L$.

Recall also from Proposition \ref{prop:R-K-L-dim-n} that $H^0\left(K_Z\right)$ is naturally isomorphic to the direct sum of $R_{K_X\otimes L}$ and $H^0\left(\Omega^1_X\right)$. The latter subspace is mapped to zero by any $\xi_{\tau}$ with $\left[\tau\right]\in R_L$, since it is preserved in all embedded deformations of $Z$. Thus, the IVHS $\xi_{\tau}\cdot$ is completely determined by its restriction to $R_{K_X\otimes L}$. Composing this restriction with the dual of the inclusion $\nu\colon R_{K_X\otimes L}\hra H^0\left(K_Z\right)$, $\xi_{\tau}\cdot$ induces a map
\begin{equation} \label{eq:xi-tau-in-R}
R_{K_X\otimes L}\stackrel{\nu}{\hra} H^0\left(K_Z\right)\stackrel{\xi_{\tau}\cdot}{\longra} H^1\left(\caO_Z\right)\stackrel{\nu^*}{\twoheadrightarrow} R_{K_X\otimes L}^*.
\end{equation}
Since both source and target spaces, as well as $\xi_{\tau}$, depend on the generalized Jacobian rings, it is natural to expect  that the map is also encoded on the generalized Jacobian rings. This is actually true, as we will see in Corollary \ref{cor:cup-product-Jacobian-rings}, but in order to prove it, we need the following technical result.

\begin{proposition} \label{prop:cup-product-X-Z}
Let $\tau\in H^0\left(L\right)$ with associated Kodaira-Spencer class $\xi_{\tau}\in H^1\left(T_Z\right)$. Let also $\gamma\in H^0\left(K_X^2\otimes L^2\right)$ and consider its restriction $\gamma_{\mid Z}\in H^0\left(K_X^2\otimes L^2_{\mid Z}\right)\cong H^0\left(K_Z^2\right)$. Then the product
\begin{equation} \label{eq:prop-cup-product-X-Z-1}
\xi_{\tau}\cdot\gamma_{\mid Z}\in H^1\left(T_Z\otimes K_Z^2\right)\cong H^1\left(K_Z\right)
\end{equation}
corresponds to
\begin{equation} \label{eq:prop-cup-product-X-Z-2}
\left(\delta\sigma\right)^2\cdot\tau\cdot\gamma\in H^2\left(\det\left(\caE_{\sigma}\otimes L^{\vee}\right)\otimes L\otimes\left(K_X^2\otimes L^2\right)\right)\cong H^2\left(K_X\right)
\end{equation}
under the natural connecting (iso)morphism $H^1\left(K_Z\right)\to H^2\left(K_X\right)$.
\end{proposition}
\begin{proof}
    Consider first the commutative diagram with exact rows
    \begin{equation} \label{diag:proof-cup-product-X-Z-1}
        \xymatrix{0 \ar[r] & \caE_{\sigma}\cong T_X\left(-\log Z\right) \ar[r] \ar@{->>}[d] & \Sigma_L \ar[r] \ar@{->>}[d] & L \ar[r] \ar@{->>}[d] & 0 \\
        0 \ar[r] & T_Z \ar[r] & T_{X\mid Z} \ar[r] & N_{Z/X}\cong L_{\mid Z} \ar[r] & 0},
    \end{equation}
    where the central vertical arrow is the composition of the surjection $\Sigma_L\thra T_X$ and the restriction $T_X\thra T_{X\mid Z}$. The first connecting homomorphisms of both rows fit in the following commutative diagram which summarizes the definition of $\xi_{\tau}$:
    \begin{equation} \label{diag:proof-cup-product-X-Z-2}
        \xymatrix{H^0\left(L\right) \ar[r] \ar[d] & H^1\left(\caE_{\sigma}\right) \ar[d] & & \tau \ar@{|->}[r] \ar@{|->}[d] & \left(\delta\sigma\right)\cdot\tau \ar@{|->}[d] \\
        H^0\left(L_{\mid Z}\right) \ar[r] & H^1\left(T_Z\right) & & \tau_{\mid Z} \ar@{|->}[r] & \xi_{\tau}}
    \end{equation}
    Tensoring with $K_X^2\otimes L^2$ gives the following diagram
    \begin{equation} \label{diag:proof-cup-product-X-Z-3}
        \xymatrix{0 \ar[r] & \caE_{\sigma}\otimes K_X^2\otimes L^2 \ar[r] \ar@{->>}[d] & \Sigma_L\otimes K_X^2\otimes L^2 \ar[r] \ar@{->>}[d] & K_X^2\otimes L^3 \ar[r] \ar@{->>}[d] & 0 \\
        0 \ar[r] & T_Z\otimes K_Z^2 \cong K_Z \ar[r] & T_{X\mid Z}\otimes K_Z^2 \ar[r] & K_Z^2\otimes L_{\mid Z} \ar[r] & 0}
    \end{equation}
    which induces the analogous diagram to \eqref{diag:proof-cup-product-X-Z-2} (in the right square we have multiplied by $\gamma$),
    \begin{equation} \label{diag:proof-cup-product-X-Z-4}
        \xymatrix{H^0\left(K_X^2\otimes L^3\right) \ar[r] \ar[d] & H^1\left(\caE_{\sigma}\otimes K_X^2\otimes L^2\right) \ar[d] & & \tau\cdot\gamma \ar@{|->}[r] \ar@{|->}[d] & \left(\delta\sigma\right)\cdot\tau\cdot\gamma \ar@{|->}[d] \\
        H^0\left(K_Z^2\otimes L_{\mid Z}\right) \ar[r] & H^1\left(K_Z\right) & & \tau_{\mid Z}\cdot\gamma_{\mid Z} \ar@{|->}[r] & \xi_{\tau}\cdot\gamma_{\mid Z}        }
    \end{equation}
    Consider now \eqref{eq:ses-Koszul-dsigma} with $r=2=\rk\caE_{\sigma}$ and tensorize it with $K_X^2\otimes L^3$ to obtain (using \eqref{eq:det-Esigma}):
    \begin{equation} \label{eq:proof-cup-product-X-Z-5}
        0 \longra \det\caE_{\sigma}\otimes K_X^2\otimes L\cong K_X \longra \bigwedge^2\Sigma_L\otimes K_X^2\otimes L \longra \caE_{\sigma}\otimes K_X^2\otimes L^2\longra 0.
    \end{equation}
    From the projection $\Sigma_L\thra T_X$ we can obtain a map $\bigwedge^2\Sigma_L\thra\det T_X=K_X^{\vee}$, which tensored by $K_X^2\otimes L$ fits into the following diagram
    \begin{equation} \label{diag:proof-cup-product-X-Z-6}
        \xymatrix{0 \ar[r] & K_X \ar[r] \ar@{=}[d] & \bigwedge^2\Sigma_L\otimes K_X^2\otimes L \ar[r] \ar@{->>}[d] & \caE_{\sigma}\otimes K_X^2\otimes L^2 \ar[r] \ar@{->>}[d] & 0 \\
        0 \ar[r] & K_X \ar[r] ^{\sigma}& K_X\otimes L \ar[r] & K_Z \ar[r] & 0}
    \end{equation}
    where the rightmost morphism is the same as the left one in \eqref{diag:proof-cup-product-X-Z-3}. We have thus commutative diagrams
    \begin{equation} \label{diag:proof-cup-product-X-Z-7}
        \xymatrix{H^1\left(\caE_{\sigma}\otimes K_X^2\otimes L^2\right) \ar[r] \ar[d] & H^2\left(K_X\right) \ar[d] & & \left(\delta\sigma\right)\cdot\tau\cdot\gamma \ar@{|->}[r] \ar@{|->}[d] & \left(\delta\sigma\right)^2\cdot\tau\cdot\gamma \ar@{=}[d] \\
        H^1\left(K_Z\right) \ar[r] & H^2\left(K_X\right) & & \xi_{\tau}\cdot\gamma_{\mid Z} \ar@{|->}[r] & \left(\delta\sigma\right)^2\cdot\tau\cdot\gamma }
    \end{equation}
    The claim follows combining the lower rows in the right squares of \eqref{diag:proof-cup-product-X-Z-4} and \eqref{diag:proof-cup-product-X-Z-7}.
\end{proof}

In order to write the map \eqref{eq:xi-tau-in-R} in terms of the generalized Jacobian rings, we need a bit more notation. Recall that for any $M\in\Pic\left(X\right)$ we have the duality map \eqref{eq:duality} $\lambda_M\colon R_M\to R_{M'}^*$, where $M'=M^{\vee}\otimes K_X^2\otimes L^3$.

\begin{definition} \label{df:phi}
For two line bundles $M,N$ on the \emph{surface} $X$ and a section $\alpha\in H^0\left(N\right)$, we define the linear map $\varphi_M\left(\alpha\right)\colon R_M\to R_{\left(M\otimes N\right)'}^*=R_{K_X^2\otimes L^3\otimes N^{\vee}\otimes M^{\vee}}^*$ given as the composition
\begin{equation} \label{eq:phi}
\varphi_M\left(\alpha\right)\colon R_M\stackrel{\cdot\alpha}{\longra} R_{M\otimes N}\stackrel{\lambda_{M\otimes N}}{\longra} R_{K_X^2\otimes L^3\otimes M^{\vee}\otimes N^{\vee}}^*
\end{equation}
of the multiplication by $\alpha$ and the duality map of $M\otimes N$.
\end{definition}

\begin{remark} \label{rmk:phi-self-dual}
It follows at once from the definition that
\begin{equation} \label{eq:phi-dual}
\left(\varphi_M\left(\alpha\right)\right)^*=\varphi_{K_X^2\otimes L^3\otimes M^{\vee}\otimes N^{\vee}}\left(\alpha\right).
\end{equation}
\end{remark}

With this last piece of notation, we can state the announced result:

\begin{corollary} \label{cor:cup-product-Jacobian-rings}
Let $\tau\in H^0\left(L\right)$ represent an embedded first-order deformation of $Z\in\left|L\right|$ with Kodaira-Spencer class $\xi_{\tau}:=KS\left(\tau\right)\in H^1\left(T_Z\right)$. Then the cup-product with $\xi_{\tau}$ fits into the following commutative diagram
\begin{equation} \label{diag:comm-diagr-cup-product}
    \xymatrix{H^0\left(K_Z\right) \ar[rr]^{\xi_{\tau}\cdot} & & H^1\left(\caO_Z\right) \ar@{->>}[d]^{\nu^*} \\
    R_{K_X\otimes L} \ar@{^(->}[u]^{\nu} \ar[rr]^{\varphi_{K_X\otimes L}\left(\tau\right)} & & R_{K_X\otimes L}^*}
\end{equation}
In particular, $\rk\varphi_{K_X\otimes L}\left(\tau\right)\leq \rk\left(\xi_{\tau}\cdot\right).$
\end{corollary}
\begin{proof}
    Recall from \eqref{eq:phi} that $\varphi_{K_X\otimes L}\left(\tau\right)$ is the composition
    \begin{equation} \label{eq:proof-comm-diagr-cup-product-1}
        R_{K_X\otimes L}\stackrel{\tau\cdot}{\longra} R_{K_X\otimes L^2}\stackrel{\lambda_{K_X\otimes L^2}}{\longra} R_{K_X\otimes L}^*
    \end{equation}
    of the multiplication map by $\tau$ and the duality map $\lambda_{K_X\otimes L^2}$. Moreover, the latter is defined by the natural pairing $R_{K_X\otimes L}\otimes R_{K_X\otimes L^2}\to R_{K_X^2\otimes L^3}\stackrel{\tr}{\longra} H^2\left(K_X\right)\cong\mC$.

    Thus, the commutativity of \eqref{diag:comm-diagr-cup-product} is equivalent to
    \begin{equation} \label{eq:proof-comm-diagr-cup-product-2}
    \left\langle\varphi_{K_X\otimes L}\left(\tau\right)\left(\alpha\right),\beta\right\rangle=\left\langle\nu^*\left(\nu\left(\alpha\right)\cdot\xi_{\tau}\right),\beta\right\rangle
    \end{equation}
    for every $\alpha,\beta\in R_{K_X\otimes L}$ (or even their representatives $\alpha,\beta\in H^0\left(K_X\otimes L\right)$) where with $\left\langle-,-\right\rangle$ we denotes the natural pairing between $R_{K_X\otimes L}$ and its dual $R_{K_X\otimes L}^*$.

    On the one hand, by \eqref{eq:natural-trace-map} and \eqref{eq:phi}, the left hand side of \eqref{eq:proof-comm-diagr-cup-product-2} is identified with the class
    \begin{equation} \label{eq:proof-comm-diagr-cup-product-3}
        \tr\left(\tau\alpha\beta\right)=\left(\delta\sigma\right)^2\cdot\tau\alpha\beta\in H^2\left(K_X\right),
    \end{equation}
    which by Proposition \ref{prop:cup-product-X-Z} corresponds to $\xi_{\tau}\cdot\left(\alpha\beta\right)_{\mid Z}$, under the isomorphism $H^2\left(K_X\right)\cong H^1\left(K_Z\right)$.

    On the other hand, the right hand side of \eqref{eq:proof-comm-diagr-cup-product-2} corresponds to the class
    \begin{equation} \label{eq:proof-comm-diagr-cup-product-4}
        \nu\left(\alpha\right)\cdot\xi_{\tau}\cdot\nu\left(\beta\right)=\xi_{\tau}\alpha_{\mid Z}\beta_{\mid Z}\in H^1\left(K_Z\otimes T_Z\otimes K_Z\right)=H^1\left(K_Z\right).
    \end{equation}
    which clearly agrees with $\xi_{\tau}\cdot\left(\alpha\beta\right)_{\mid Z}$.
\end{proof}

We now prove a generalization of \cite[Lemma 2.4]{FP-plane-curves}, giving a particular property of sections in the kernel of a general $\xi_{\tau}\cdot$, which allows us to tackle the problem directly on $R_{K_X\otimes L}$ without need to consider the elements $\tau\in R_L$.

\begin{lemma} \label{lem:alpha2}
Let $\tau \in H^0\left(L\right)$ be such that $\varphi_{K_X\otimes L}\left(\tau\right)\colon R_{K_X\otimes L}\to R_{K_X\otimes L}^*$ has maximal rank (among all possible $\tau\in H^0\left(L\right)$). Then for every $\alpha\in\ker\varphi_{K_X\otimes L}\left(\tau\right)$ it holds $\lambda_{K_X^2\otimes L^2}\left(\alpha^2\right)=0\in R_L^*$.
\end{lemma}
\begin{proof}
The condition $\alpha\in\ker\varphi_{K_X\otimes L}\left(\tau\right)$ means that $\lambda_{K_X\otimes L^2}\left(\tau\cdot\alpha\right)=0\in R_{K_X\otimes L}^*$, i.e.
\begin{equation} \label{eq:condition-alpha2}
\tr\left(\tau\cdot\alpha\cdot\beta\right)=0\in H^2\left(K_X\right)
\end{equation}
for every $\beta\in R_{K_X\otimes L}$.

Assuming this, we want to prove that $\tr\left(\alpha^2\cdot w\right)=0$ for every $w\in R_L$.

Let $w\in R_L$ be arbitrary. By semicontinuity, the rank of $\varphi_{K_X\otimes L}\left(\tau\right)$ is lower-semicontinuous, i.e. the maximum is attained for $\tau$ on a dense open subset of $H^0\left(L\right)$. In particular, there is a small neighbourhood $U\subseteq \mC$ of $0\in\mC$ such that $\varphi_{K_X\otimes L}\left(\tau+tw\right)$ is of maximal rank for all $t\in U$. Moreover, when $t$ varies, the kernels of $\varphi_{K_X\otimes L}\left(\tau+tw\right)$ form a (holomorphic) vector subbundle of the trivial vector bundle with fibre $R_{K_X\otimes L}$ over $U$. Thus, $\alpha\in\ker\varphi_{K_X\otimes L}\left(\tau \right)$ can be extended to a (holomorphic) section over $U$ with $\alpha\left(t\right)\in\ker\varphi_{K_X\otimes L}\left(\tau+tw\right)\subseteq R_{K_X\otimes L}$ for every $t\in U$ and $\alpha\left(0\right)=\alpha$. This means
\begin{equation} \label{eq:section-kernel-1}
\varphi_{K_X\otimes L}\left(\tau+tw\right)\left(\alpha\left(t\right)\right)=0\in R_{K_X\otimes L}^*\quad\forall\,t\in U,
\end{equation}
or in other words
\begin{equation} \label{eq:section-kernel-2}
\tr\left(\left(\tau+tw\right)\cdot\alpha\left(t\right)\cdot\beta\right)=0\in H^2\left(K_X\right)\cong\mC\quad\forall\,t\in U,\forall\,\beta\in R_{K_X\otimes L}.
\end{equation}
Differentiating \eqref{eq:section-kernel-2} and evaluating at $t=0$, we obtain
\begin{equation} \label{eq:section-kernel-3}
\tr\left(w\cdot\alpha\left(0\right)\cdot\beta+\tau\cdot\alpha'\left(0\right)\cdot\beta\right)=0\quad\forall\,\beta\in R_{K_X\otimes L}.
\end{equation}
In particular, for $\beta=\alpha=\alpha\left(0\right)$ we have
\begin{equation} \label{eq:section-kernel-4}
\tr\left(w\cdot\alpha^2\right)=-\tr\left(\tau\cdot\alpha'\left(0\right)\cdot\alpha\right)=0,
\end{equation}
where the last equality follows from \eqref{eq:condition-alpha2} with $\beta=\alpha'\left(0\right)$.
\end{proof}

In view of Lemma \ref{lem:alpha2}, one way to prove that the general first-order embedded deformation of $Z$ has maximal infinitesimal variation is to prove that the only $\alpha\in R_{K_X\otimes L}$ with $\lambda_{K_X^2\otimes L^2}\left(\alpha^2\right)=0\in R_L^*$ is $\alpha=0$.

\begin{remark} \label{rmk:approach-quadrics}
To prove the injectivity of $\varphi_{K_X\otimes L}\left(\tau\right)$ for a general $\tau\in R_L$, this approach can only work if $\dim R_{K_X\otimes L}\leq \dim R_L$, which is not expected if $K_X$ is positive. Indeed, for a given $\tau\in R_L$, the condition $\tr\left(\tau\cdot\alpha^2\right)=0$ defines a quadric $Q_{\tau}\subseteq \mP\left(R_{K_X\otimes L}\right)$. Thus, $\lambda_{K_X^2\otimes L^2}\left(\alpha^2\right)=0$ is equivalent to $\alpha$ lying on the intersection $Q_{\tau_1}\cap\ldots\cap Q_{\tau_m}$ for $\tau_1,\ldots,\tau_m\in R_L$ a basis. If $\dim R_L=m\leq \dim\mP\left(R_{K_X\otimes L}\right)=\dim R_{K_X\otimes L}-1$ the intersection won't be empty, hence there will be non-zero $\alpha$'s with $\lambda_{K_X^2\otimes L^2}\left(\alpha^2\right)=0\in R_L^*$.

Thus, we only expect this approach to work for curves on surfaces with somehow non-positive canonical bundle, such as del Pezzo surfaces or at most K3 or Abelian surfaces.
\end{remark}

It will actually be more convenient to work with the representatives $\alpha\in H^0\left(K_X\otimes L\right)$ (which we still denote as $\alpha$ for the sake of simplicity), and thus consider the elements in the kernel of the composition
\begin{equation} \label{eq:definition-lambda}
    \lambda\colon H^0\left(K_X^2\otimes L^2\right)\longra R_{K_X^2\otimes L^2}\stackrel{\lambda_{K_X^2\otimes L^2}}{\longra}R_L^*
\end{equation}
Our aim is then to prove that the only $\alpha\in H^0\left(K_X\otimes L\right)$ with $\lambda\left(\alpha^2\right)=0$ are those lying on $J_{K_X\otimes L}$. This leads to the following definition.

\begin{definition} \label{df:D}
Let $D\subseteq\mP\left(H^0\left(K_X^2\otimes L^2\right)\right)$ be defined as
\begin{equation} \label{eq:D}
D=\left\{\left[\alpha\beta\right]\in \mP\left(H^0\left(K_X^2\otimes L^2\right)\right)\,\mid\,\alpha,\beta\in H^0\left(K_X\otimes L\right), \alpha\beta\in \ker\lambda\right\},
\end{equation}
where from now on we use square brackets to denote the corresponding point in the projectivization of a vector space.
\end{definition}
In other words, $D$ is the intersection of the linear subspace $\mP\left(\ker\lambda\right)$ and the image of the Segre-multiplication map $\mP\left(H^0\left(K_X\otimes L\right)\right)\times\mP\left(H^0\left(K_X\otimes L\right)\right)\to \mP\left(H^0\left(K_X^2\otimes L^2\right)\right)$.

\begin{theorem} \label{thm:lower-bound-dim-D}
Assume $L$ has not maximal infinitesimal variation at a smooth $Z\in\left|L\right|$. Then it holds
\begin{equation} \label{eq:lower-bound-dim-D}
\dim D\geq h^0\left(K_X\otimes L\right)+p_g\left(X\right)-1.
\end{equation}
\end{theorem}
\begin{proof}
The proof follows some ideas of \cite[Proposition 3.4]{FP-plane-curves}. Define
\begin{equation} \label{eq:Y}
Y:=\left\{\left[\alpha\right]\in\mP\left(H^0\left(K_X\otimes L\right)\right)\,\mid\,\alpha^2\in \ker\lambda\right\}
\end{equation}
and the following incidence variety
\begin{equation} \label{eq:Z}
W:=\left\{\left(\left[\alpha\right],\left[\beta\right]\right)\in Y\times\mP\left(H^0\left(K_X\otimes L\right)\right)\,\mid\,\alpha\beta\in \ker\lambda\right\}.
\end{equation}
The projection onto the first factor defines a surjective morphism $\pi\colon W\to Y$ ($\beta=\alpha$ gives a preimage for any $\left[\alpha\right]\in Y$), and multiplication defines a morphism $\psi\colon W\to D$.

Note first that $\psi$ is a finite morphism: the zero locus of any element $\left[\gamma\right]\in D\subseteq\mP\left(H^0\left(K_X^2\otimes L^2\right)\right)$ has finitely many decompositions into the sum of two effective divisors, hence in particular $\gamma$ admits finitely many decompositions (up to multiplication of one factor by $a\in\mC^*$) of the form $\gamma=\alpha\beta$ with $\alpha,\beta\in H^0\left(K_X\otimes L\right)$. Thus,
\begin{equation} \label{eq:W-D}
\dim W\leq \dim D.
\end{equation}
We consider now the fibres of $\pi$: for $\left[\alpha\right]\in Y$, it holds
\begin{align} \label{eq:fibres-pi}
\pi^{-1}\left(\left[\alpha\right]\right)&=\left\{\left[\alpha\right]\right\}\times\left\{\left[\beta\right]\in\mP\left(H^0\left(K_X\otimes L\right)\right)\,\mid\,\alpha\beta\in \ker\lambda\right\} \\
&\cong\mP\left(\ker\left(H^0\left(K_X\otimes L\right)\stackrel{\cdot\alpha}{\longra}H^0\left(K_X^2\otimes L^2\right)\stackrel{\lambda}{\longra} R_L^*\right)\right)
\end{align}
The latter composition factorizes also as $H^0\left(K_X\otimes L\right)\twoheadrightarrow R_{K_X\otimes L}\stackrel{\varphi_{K_X\otimes L}\left(\alpha\right)}{\longra} R_L^*$, where the first map is surjective, and thus
\begin{align}
\dim\pi^{-1}\left(\left[\alpha\right]\right)&=\dim\ker\left(H^0\left(K_X\otimes L\right)\twoheadrightarrow R_{K_X\otimes L}\stackrel{\varphi_{K_X\otimes L}\left(\alpha\right)}{\longra} R_L^*\right)-1
\\
&=h^0\left(K_X\otimes L\right)-\rk\varphi_{K_X\otimes L}\left(\alpha\right)-1 \label{eq:dim-fibres-pi} \\
&\stackrel{\eqref{eq:phi-dual}}{=}h^0\left(K_X\otimes L\right)-\rk\varphi_L\left(\alpha\right)-1
\end{align}
Since $\pi\colon W\to Y$ is a surjective morphism, it holds
\begin{equation} \label{eq:dim-W-2}
\dim W=\dim Y+\dim\pi^{-1}\left(\left[\alpha\right]\right)=\dim Y+h^0\left(K_X\otimes L\right)-\rk\varphi_L\left(\alpha\right)-1
\end{equation}
for a general $\left[\alpha\right]\in Y$.

It remains to find a lower bound on the dimension of $Y$. To this aim we introduce one further incidence variety. Denote by $\lambda'$ the composition
\begin{equation} \label{eq:lambda'}
\lambda'\colon H^0\left(K_X\otimes L^2\right)\twoheadrightarrow R_{K_X\otimes L^2}\stackrel{\lambda_{K_X\otimes L^2}}{\longra} R_{K_X\otimes L}^*,
\end{equation}
let $U\subseteq\mP\left(H^0\left(L\right)\right)$ be the dense open subset consisting of the $\left[\tau\right]$ for which $\varphi_{K_X\otimes L}\left(\tau\right)$ has maximal rank, and set
\begin{equation} \label{eq:I-lower-bound}
I:=\left\{\left(\left[\tau\right],\left[\alpha\right]\right)\in U\times\mP\left(H^0\left(K_X\otimes L\right)\right)\,\mid\, \tau\alpha\in\ker\lambda'\right\},
\end{equation}
which is equipped with two projections $\mu_1\colon I\to U$ and $\mu_2\colon I\to\mP\left(H^0\left(K_X\otimes L\right)\right)$.

The fibre of $\mu_1$ over a given $\left[\tau\right]\in U$ is the projectivization of the kernel of the composition
\begin{equation} \label{eq:fibre-mu1-1}
H^0\left(K_X\otimes L\right)\longra H^0\left(K_X\otimes L\right)/J_{K_X\otimes L}=R_{K_X\otimes L}\stackrel{\varphi_{K_X\otimes L}\left(\tau\right)}{\longra} R_{K_X\otimes L}^*,
\end{equation}
which obviously contains $J_{K_X\otimes L}=H^0\left(K_X\right)\cdot\sigma$ (Proposition \ref{prop:R-K-L-dim-n}), with equality if and only if $\varphi_{K_X\otimes L}\left(\tau\right)$ is an isomorphism. Thus, $L$ has maximal infinitesimal variation at $Z\in\left|L\right|$ if and only if the general fibre of $\mu_1$ has dimension precisely $\dim\mP\left(J_{K_X\otimes L}\right)=p_g\left(X\right)-1$. By assumption, $L$ has not maximal infinitesimal variation, and hence $\mu_1$ is surjective and its general fibre has dimension at least $p_g\left(X\right)$. Therefore, it holds
\begin{equation} \label{eq:dim-I-1}
\dim I\geq \dim U+p_g\left(X\right)=\dim\mP\left(H^0\left(L\right)\right)+p_g\left(X\right)=h^0\left(L\right)+p_g\left(X\right)-1.
\end{equation}
We now consider the projection $\mu_2$. By Lemma \ref{lem:alpha2}, for every $\left(\left[\tau\right],\left[\alpha\right]\right)\in I$ it holds $\lambda\left(\alpha^2\right)=0$, i.e. $\left[\alpha\right]\in Y$ and thus $\mu_2\left(I\right)\subseteq Y$. As for the fibers of $\mu_2$ it holds
\begin{equation} \label{eq:fibre-m2}
\mu_2^{-1}\left(\left[\alpha\right]\right)\cong\left\{\left[\tau\right]\in U\,\mid\,\varphi_L\left(\alpha\right)\left(\tau\right)=0\in R_{K_X\otimes L}^*\right\}=U\cap \mP\left(\ker\left(H^0\left(L\right)\twoheadrightarrow R_L\stackrel{\varphi_L\left(\alpha\right)}{\longra} R_{K_X\otimes L}^*\right)\right).
\end{equation}
All in all we obtain
\begin{equation} \label{eq:dim-I-2}
\dim I=\dim\mu_2\left(I\right)+\dim\mu_2^{-1}\left(\left[\alpha\right]\right)\leq\dim Y+\dim\ker\varphi_L\left(\alpha\right)+\dim J_L-1
\end{equation}
for a general $\alpha\in\mu_2\left(I\right)$. Combining \eqref{eq:dim-I-1} and \eqref{eq:dim-I-2}, we obtain
\begin{equation} \label{eq:dim-Y}
\dim Y\geq h^0\left(L\right)-\dim J_L-\dim\ker\varphi_L\left(\alpha\right)+p_g\left(X\right)=\rk\varphi_L\left(\alpha\right)+p_g\left(X\right),
\end{equation}
which with \eqref{eq:W-D} and \eqref{eq:dim-W-2} gives the wanted
\begin{equation} \label{eq:dim-W-3}
\dim D\geq \dim W\geq h^0\left(K_X\otimes L\right)+p_g\left(X\right)-1.
\end{equation}
\end{proof}

\section{Curves in $\mP^1\times\mP^1$}

\label{sect:P1-P1}

We now show that for $X=\mP^1\times\mP^1$ and $L$ sufficiently ample (with a very explicit condition on the bidegree) we can obtain an upper-bound on $\dim D$, contradicting the lower bound of Theorem \ref{thm:lower-bound-dim-D}, thus proving Theorem \ref{main-thm}.

Let us then fix $X=\mP^1\times\mP^1$. Recall that line bundles on $\mP^1\times\mP^1$ are determined by its bidegree. We denote by $\caO_X\left(d,e\right)=\caO_{\mP^1}\left(d\right)\boxtimes\caO_{\mP^1}\left(e\right)$. More generally, for any sheaf $\caF$ of $\caO_X$-modules on $X$ we denote by $\caF\left(d,e\right)=\caF\otimes\caO_X\left(d,e\right)$.

We also denote by $\left(x_0:x_1\right)$ (resp. $\left(y_0:y_1\right)$) homogeneous coordinates on the first (resp. second) factor. Consider the polynomial ring $S:=\mC\left[x_0,x_1,y_0,y_1\right]$, bigraded so that $\deg\left(x_i\right)=\left(1,0\right)$ and $\deg\left(y_i\right)=\left(0,1\right)$, and write $S_{d,e}$ for the summand of polynomials of bidegree $\left(d,e\right)$, which correspond to the global sections of $\caO_X\left(d,e\right)$. Also, for any bihomogeneous ideal $I\subseteq S$ we denote by $I_{d,e}=I\cap S_{d,e}$, the summand of elements of bidegree $\left(d,e\right)$.

Let now $L=\caO_X\left(d,e\right)$ and choose an irreducible non-zero section $F\in H^0\left(X,L\right)=S_{d,e}$. Its zero locus $Z:=\left\{F=0\right\}\subseteq\mP^1\times\mP^1$ is a smooth curve if and only if the four partial derivatives of $F$ have no common zero, i.e. $\left\{\frac{\partial F}{\partial x_0}=\frac{\partial F}{\partial x_1}=\frac{\partial F}{\partial y_0}=\frac{\partial F}{\partial y_1}=0\right\}=\emptyset\subseteq\mP^1\times\mP^1$.

To $L$ we can attach the bundle of first-order differential operators $\Sigma_L$, and then $F\in H^0\left(L\right)$ induces the surjective morphism $dF\colon\Sigma_L\to L$. For any $M=\caO_X\left(a,b\right)$ we can consider the corresponding generalized Jacobian ideal
$$J_{a,b}:=J_M=\im\left(H^0\left(\Sigma_L\otimes L^{\vee}\otimes M\right)=H^0\left(\Sigma_L\left(a-d,b-e\right)\right)\longra H^0\left(M\right)\right)\subseteq H^0\left(M\right)=S_{a,b}.$$

\begin{lemma} \label{lem:J-P1-P1}
The generalized Jacobian ideal coincides with the ``usual'' Jacobian ideal generated by the partial derivatives of $F$, i.e.
\begin{align}
J_{a,b}&=\left(\frac{\partial F}{\partial x_0},\frac{\partial F}{\partial x_1},\frac{\partial F}{\partial y_0},\frac{\partial F}{\partial y_1}\right)_{a,b}\\
&=\left\{P_0\frac{\partial F}{\partial x_0}+P_1\frac{\partial F}{\partial x_1}+Q_0\frac{\partial F}{\partial y_0}+Q_1\frac{\partial F}{\partial y_1}\,\mid\,P_0,P_1\in S_{a-d+1,b-e},Q_0,Q_1\in S_{a-d,b-e+1}\right\}.
\end{align}
\end{lemma}

In particular, the pieces of small degree are $J_{d-1,e-1}=0$,
\begin{equation} \label{eq:J-a-b-small}
J_{d-1,e}=\left\{a_0\frac{\partial F}{\partial x_0}+a_1\frac{\partial F}{\partial x_1}\,\mid\,a_0,a_1\in\mC\right\}, J_{d,e-1}=\left\{b_0\frac{\partial F}{\partial y_0}+b_1\frac{\partial F}{\partial y_1}\,\mid\,b_0,b_1\in\mC\right\},
\end{equation}
and $J_{d,e}$ is generated as a vector space by
\begin{equation} \label{eq:generators-J-d-e}
\left\{x_0\frac{\partial F}{\partial x_0},x_1\frac{\partial F}{\partial x_0},x_0\frac{\partial F}{\partial x_1},x_1\frac{\partial F}{\partial x_1},y_0\frac{\partial F}{\partial y_0},y_1\frac{\partial F}{\partial y_0},y_0\frac{\partial F}{\partial y_1},y_1\frac{\partial F}{\partial y_1}\right\}.
\end{equation}
These generators are however not linearly independent, since the Euler relations with respect to each subset of variables give
\begin{equation} \label{eq:Euler-P1xP1}
    \frac{1}{d}\left(x_0\frac{\partial F}{\partial x_0}+x_1\frac{\partial F}{\partial x_1}\right)=F=\frac{1}{e}\left(y_0\frac{\partial F}{\partial y_0}+y_1\frac{\partial F}{\partial y_1}\right).
\end{equation}

\begin{proof}[Proof of Lemma \ref{lem:J-P1-P1}]
We first describe a trivialization of $\Sigma_L$. For $i,j=0,1$ consider the affine open subset
\begin{equation} \label{eq:proof:lem:J-P1-P1-1}
\begin{array}{rcl}
U_{ij}:=\left\{\left(\left[x_0:x_1\right],\left[y_0:y_1\right]\right)\,\mid\,x_iy_j\neq 0\right\} & \stackrel{\cong}{\longlra} & \mC^2 \\
\left(\left[x_0:x_1\right],\left[y_0:y_1\right]\right) & \mapsto & \left(z_{1-i},w_{1-j}\right):=\left(\frac{x_{1-i}}{x_i},\frac{y_{1-j}}{y_j}\right)
\end{array}
\end{equation}
On $U_{ij}$ the line bundle $L$ is trivialized by the section $e_{ij}:=x_i^dy_j^e$, and any $F\left(x_0,x_1,y_0,y_1\right)\in H^0\left(L\right)$ can be uniquely written as $F=f_{ij}\left(z_{1-i},w_{1-j}\right)x_i^dy_j^e$ for a polynomial $f_{ij}\in\mC\left[z_{1-i},w_{1-j}\right]$. We can obtain $f_{ij}$ explicitly, by substituting $x_i=y_j=1$, $x_{1-i}=z_{1-i}$ and $y_{1-j}=w_{1-j}$.

Moreover, $\Sigma_L$ is trivialized by $\left\{1,D_z^{ij},D_w^{ij}\right\}$, where $D_z^{ij}\mapsto\frac{\partial}{\partial z_{1-i}}$ and $D_w^{ij}\mapsto\frac{\partial}{\partial w_{1-j}}$ under the natural projection $\Sigma_L\to T_{\mP^1\times\mP^1}$. And the morphism $dF\colon\Sigma_L\to L$ is given by $dF\left(1\right)=f_{ij}e_{ij}=F$,
\begin{equation} \label{eq:proof:lem:J-P1-P1-2}
    dF\left(D_z^{ij}\right)=\frac{\partial f_{ij}}{\partial z_{1-i}}e_{ij}=\frac{\partial F}{\partial x_{1-i}}\left(x_i=1,y_j=1\right)x_i^ey_j^d=x_i\frac{\partial F}{\partial x_{1-i}}
\end{equation}
and
\begin{equation} \label{eq:proof:lem:J-P1-P1-3}
    dF\left(D_w^{ij}\right)=\frac{\partial f_{ij}}{\partial w_{1-j}}e_{ij}=\frac{\partial F}{\partial y_{1-j}}\left(x_i=1,y_j=1\right)x_i^ey_j^d=y_j\frac{\partial F}{\partial y_{1-j}}.
\end{equation}
Since the result of \eqref{eq:proof:lem:J-P1-P1-2} is independent of $j$, we conclude that $D_z^{i0}=D_z^{i1}$ on  $U_{i0}\cap U_{i1}$, and analogously $D_w^{0j}=D_w^{1j}$ on $U_{0j}\cap U_{1j}$. We just denote then by $D_z^i$ resp. $D_w^j$ the local section of $\Sigma_L$ given by $D_z^{ij}$ on $U_{ij}$ resp. $D_w^{ij}$ on $U_{ij}$.

In order to compare $D_z^0$ and $D_z^1$ on $U_{0j}\cap U_{1j}$, we compare their images under $dF$ for an arbitrary $F\in H^0\left(L\right)$, which are respectively $x_0\frac{\partial F}{\partial x_1}$ and $x_1\frac{\partial F}{\partial x_0}$. Using the Euler relation with respect to the variables $x_0,x_1$ we obtain
\begin{equation} \label{eq:proof:lem:J-P1-P1-4}
    dF\left(x_1^2D_z^0+x_0^2D_z^1\right)=x_1^2x_0\frac{\partial F}{\partial x_1}+x_0^2x_1\frac{\partial F}{\partial x_0}=x_0x_1\left(x_0\frac{\partial F}{\partial x_0}+x_1\frac{\partial F}{\partial x_1}\right)=d\cdot x_0x_1 F =dF\left(d\cdot x_0x_1\cdot 1\right).
\end{equation}
Since this should hold for any $F\in H^0\left(L\right)$, it must hold
\begin{equation} \label{eq:proof:lem:J-P1-P1-5-1}
    x_1^2D_z^0+x_0^2D_z^1=d\cdot x_0x_1\cdot 1,\quad\text{i.e.}\quad D_z^1=d\cdot\frac{x_1}{x_0}\cdot1-\left(\frac{x_1}{x_0}\right)^2D_z^0=d\cdot z_1\cdot 1-z_1^2D_z^0.
\end{equation}
Analogously, we can prove
\begin{equation} \label{eq:proof:lem:J-P1-P1-5-3}
    D_w^1=d\cdot w_1\cdot 1-w_1^2D_w^0.
\end{equation}

We are now ready to compute a basis of the global sections of $\Sigma_L\left(a,b\right)$ for any $a,b\in\mZ$.

Note first that the short exact sequence
\begin{equation} \label{eq:proof:lem:J-P1-P1-5-5}
    0 \longra \caO_X\left(a,b\right) \longra \Sigma_L\left(a,b\right)\longra T_X\left(a,b\right)=\caO_X\left(a+2,b\right)\oplus \caO_X\left(a,b+2\right)\longra 0
\end{equation}
gives $h^0\left(\Sigma_L\left(a,b\right)\right)=0$ unless $a,b\geq -2$. Thus from now on we restrict ourselves to this case.

On $U_{ij}$, $\Sigma_L\left(a,b\right)$ is trivialized by $1\cdot x_i^ay_j^b, D_z^i\otimes x_i^1y_j^b$ and $D_w^j\otimes x_i^ay_j^b$, and the gluing relations are given by \eqref{eq:proof:lem:J-P1-P1-5-1} and \eqref{eq:proof:lem:J-P1-P1-5-3}.

On $U_{ij}$ a global section of $\Sigma_L\left(a,b\right)$ is thus described as
\begin{equation} \label{eq:proof:lem:J-P1-P1-6}
    \left(A_{ij}\cdot 1+B_{ij}D_z^i+C_{ij}D_w^j\right)\otimes x_i^ay_j^b
\end{equation}
for certain polynomials $A_{ij},B_{ij},C_{ij}\in\mC\left[z_{1-j},w_{1-j}\right]$. We will now use the relations \eqref{eq:proof:lem:J-P1-P1-5-1} and \eqref{eq:proof:lem:J-P1-P1-5-3} to find relations between these polynomials, as well as upper bounds on their (bi)degrees.

Using \eqref{eq:proof:lem:J-P1-P1-4} and $z_0=\frac{x_0}{x_1}=\frac{1}{z_1}$ we obtain
\begin{align}
    A_{0j}\left(z_1,w_{1-j}\right)&=\left[A_{1j}\left(\frac{1}{z_1},w_{1-j}\right)+d\cdot z_1B_{1j}\left(\frac{1}{z_1},w_{1-j}\right)\right]z_1^a \label{eq:proof:lem:J-P1-P1-7}
    \\
    B_{0j}\left(z_1,w_{1-j}\right)&=-B_{1j}\left(\frac{1}{z_1},w_{1-j}\right)z_1^{a+2} \label{eq:proof:lem:J-P1-P1-8}
    \\
    C_{0j}\left(z_1,w_{1-j}\right)&=C_{1j}\left(\frac{1}{z_1},w_{1-j}\right)z_1^a. \label{eq:proof:lem:J-P1-P1-9}
\end{align}
From \eqref{eq:proof:lem:J-P1-P1-8} and \eqref{eq:proof:lem:J-P1-P1-9} it follows $\deg_{z_0}B_{1j}\leq a+2$ and $\deg_{z_0}C_{1j}\leq a$.

By symmetry, it must also hold
\begin{align}
    A_{i0}\left(z_{1-i},w_1\right)&=\left[A_{i1}\left(z_{1-i},\frac{1}{w_1}\right)+e\cdot w_1C_{i1}\left(z_{1-i},\frac{1}{w_1}\right)\right]w_1^b
    \label{eq:proof:lem:J-P1-P1-10}
    \\
    B_{i0}\left(z_{1-i},\frac{1}{w_1}\right)&=-B_{i1}\left(z_{1-i},\frac{1}{w_1}\right)w_1^{b+2}
    \label{eq:proof:lem:J-P1-P1-11}
    \\
    C_{i0}\left(z_{1-i},\frac{1}{w_1}\right)&=C_{i1}\left(z_{1-i},\frac{1}{w_1}\right)w_1^b,
    \label{eq:proof:lem:J-P1-P1-12}
\end{align}
hence, in particular, $\deg_{w_0}B_{1j}\leq b$ and $\deg_{w_0}C_{1j}\leq b+2$.

We can try with monomials $B_{11}=z_0^rw_0^s$ and $C_{11}=z_0^nw_0^m$, where $0\leq r\leq a+2$, $0\leq s\leq b$, $0\leq n\leq a$ and $0\leq m\leq b+2$, and check the possibilities for $A_{11}$.

From \eqref{eq:proof:lem:J-P1-P1-7} it follows that $A_{11}\left(\frac{1}{z_1},w_0\right)z_1^a+d\cdot z_1^{a+1-r}w_0^s$ must be a polynomial on $z_1$ and $w_0$.
\begin{itemize}
    \item If $r\leq a+1$ it is enough that $\deg_{z_0}A_{11}\leq a$, thus we can take any monomial $A_{11}\left(z_0,w_0\right)=z_0^tw_0^u$ with $0\leq t\leq a$.
    \item If $r=a+2$, $A_{11}$ must contain the monomial $-d\cdot z_0^{a+1}w_0^s$, in order to cancel the term $d\cdot z_1^{-1}w_0^s$ (and also any monomial as in the case $r\leq a+1$).
\end{itemize}
By symmetry:
\begin{itemize}
    \item If $m\leq b+1$, $A_{11}$ can contain any monomial $z_0^tw_0^u$ with $0\leq u\leq b$,
    \item If $m=\leq b+2$, then $A_{11}$ must contain $-e\cdot z_0^nw_0^{b+2}$.
\end{itemize}
Altogether, we have the following basis for $H^0\left(\Sigma_L\left(a,b\right)\right)$ (we give their local expressions on $U_{11}$)
\begin{enumerate}
    \item $z_0^tw_0^u\cdot 1\otimes x_1^ay_1^b$ with $0\leq t\leq a$ and $0\leq u\leq b$, \label{item:basis-1}
    \item $z_0^rw_0^s\cdot D_z^1\otimes x_1^ay_1^b$ with $0\leq r\leq a+1$ and $0\leq s\leq b$, \label{item:basis-2}
    \item $\left(-d\cdot z_0^{a+1}w_0^s\cdot 1+z_0^{a+2}w_0^s\cdot D_z^1\right)\otimes x_1^ay_1^b$ with $0\leq s\leq b$ \label{item:basis-3}
    \item $z_0^nw_0^m\cdot D_w^1\otimes x_1^ay_1^b$ with $0\leq n\leq a$ and $0\leq m\leq b+1$, and \label{item:basis-4}
    \item $\left(-e\cdot z_0^nw_0^{b+1}\cdot 1+z_0^nw_0^{b+2}\cdot D_w^1\right)\otimes x_1^ay_1^b$ with $0\leq n\leq a$. \label{item:basis-5}
\end{enumerate}

Note that if $a=-2$ or $b=-2$ there are no elements in the above list, hence also $H^0\left(\Sigma_L\left(a,b\right)\right)=0$. This is compatible with the short exact sequence \eqref{eq:proof:lem:J-P1-P1-5-5}, since in this case we have $H^1\left(\caO_X\left(a,b\right)\right)\neq 0$, and the connecting homomorphism turns out to be an isomorphism.

In order to finish the proof, we just need to compute a set of generators of
\begin{equation} \label{eq:proof:lem:J-P1-P1-13}
J_{a,b}=\im\left(dF\colon H^0\left(\Sigma_L\left(a-d,b-e\right)\right)\to H^0\left(\caO_X\left(a,b\right)\right)\right)
\end{equation}
by applying the differential  $dF$ to the elements of the above basis (replacing $a$ by $a-d$ and $b$ by $b-d$). Since $H^0\left(\Sigma_L\left(a-d,b-e\right)\right)=0$ unless $a-d,b-e\geq -1$, we can assume $a\geq d-1$ and $b\geq e-1$. Recall that $dF\left(1\right)=F$, $dF\left(D_z^1\right)=x_1\frac{\partial F}{\partial x_0}$ and $dF\left(D_w^1\right)=y_1\frac{\partial F}{\partial y_0}$, and thus
\begin{enumerate}
    \item $dF\left(z_0^tw_0^u\cdot 1\otimes x_1^{a-d}y_1^{b-e}\right)=\frac{x_0^t}{x_1^t}\frac{y_0^u}{y_1^u} x_1^{a-d}y_1^{b-e} F=x_0^tx_1^{a-d-t}y_0^uy_1^{b-e-u}F$ with $0\leq t\leq a-d$ and $0\leq u\leq b-e$, i.e. we obtain all possible multiples of $F$ with bidegree $\left(a,b\right)$; \label{item:im-basis-1}
    \item $dF\left(z_0^rw_0^s\cdot D_z^1\otimes x_1^{a-d}y_1^{b-e}\right)=x_0^rx_1^{a-d-r+1}y_0^sy_1^{b-e-s}\frac{\partial F}{\partial x_0}$ with $0\leq r\leq a-d+1$ and $0\leq s\leq b-e$, i.e. all multiples of $\frac{\partial F}{\partial x_0}$ of bidegree $\left(a,b\right)$; \label{item:im-basis-2}
    \item $dF\left(\left(-d\cdot z_0^{a-d+1}w_0^s\cdot 1+z_0^{a-d+2}w_0^s\cdot D_z^1\right)\otimes x_1^{a-d}y_1^{b-e}\right)=$ \label{item:im-basis-3}

    $-d\cdot x_0^{a-d+1}x_1^{-1}y_0^sy_1^{b-e-s}F+x_0^{a-d+2}x_1^{-2}y_0^sy_1^{b-e-s}x_1\frac{\partial F}{\partial x_0}=$

    $x_0^{a-d+1}x_1^{-1}y_0^sy_1^{b-e-s}\left[-d\cdot F+x_0\frac{\partial F}{\partial x_0}\right]=x_0^{a-d+1}x_1^{-1}y_0^sy_1^{b-e-s}\left[-x_1\frac{\partial F}{\partial x_1}\right]=-x_0^{a-d+1}y_0^sy_1^{b-e-s}\frac{\partial F}{\partial x_1}$ with $0\leq s\leq b-e$, i.e. all multiples of $x_0^{a-d+1}\frac{\partial F}{\partial x_1}$ of bidegree $\left(a,b\right)$. The remaining multiples of $\frac{\partial F}{\partial x_1}$ of the same bidegree, admitting powers of $x_1$, can be obtained from \eqref{item:im-basis-1} and \eqref{item:im-basis-2} using the Euler relation with respect to $x_0,x_1$.

    \item $dF\left(z_0^nw_0^m\cdot D_w^1\otimes x_1^{a-d}y_1^{b-e}\right)=x_0^nx_1^{a-d-n}y_0^my_1^{b-e-m+1}\frac{\partial F}{\partial y_0}$ with $0\leq n\leq a-d$ and $0\leq m\leq b-e+1$, and \label{item:im-basis-4}

    \item $dF\left(\left(-e\cdot z_0^nw_0^{b-e+1}\cdot 1+z_0^nw_0^{b-e+2}\cdot D_w^1\right)\otimes x_1^{a-d}y_1^{b-e}\right)=-x_0^nx_1^{a-d-n}y_1^{b-e+1}\frac{\partial F}{\partial y_1}$ with $0\leq n\leq a-d$. \label{item:im-basis-5}
\end{enumerate}
Thus $J_{a,b}$ contains all products of the partial derivatives of $F$ and monomials (of the right bidegree), which proves the claim.
\end{proof}

In order to find a contradiction with Theorem \ref{thm:lower-bound-dim-D}, we need to understand the geometry of the linear system on $Z$ induced by $J_{d,e}$. More strictly speaking, since $F\in J_{d,e}$, the restriction map $J_{d,e}\subseteq H^0\left(L\right)\to H^0\left(L_{\mid Z}\right)=H^0\left(\caO_Z\left(d,e\right)\right)$ is not injective, but its image induces a linear subsystem of $\left|\caO_Z\left(d,e\right)\right|$ that we will still denote $\left|J_{d,e\mid Z}\right|$. It will be helpful to consider also the restrictions of $J_{d-1,e}$ and $J_{d,e-1}$, which by degree reasons inject in $H^0\left(\caO_Z\left(d-1,e\right)\right)$ and $H^0\left(\caO_Z\left(d,e-1\right)\right)$ respectively.

\begin{theorem} \label{thm:J-d-e}
Suppose $d,e\geq 2$. Then:
\begin{enumerate}
    \item The base divisor of $\left|J_{d-1,e\mid Z}\right|$ coincides with the ramification divisor of the second projection $Z\hra X=\mP^1\times\mP^1\stackrel{\pi_2}{\to}\mP^1$, and after removing the base points it induces the projection of $Z$ onto the first factor.
    \item The restricted linear system $\left|J_{d,e\mid Z}\right|$ is base point free and the induced morphism $Z\to\left|J_{d,e\mid Z}\right|^{\vee}$ is birational into its image.
\end{enumerate}
\end{theorem}
\begin{proof}
    \begin{enumerate}
        \item Since $\left\{\frac{\partial F}{\partial x_0},\frac{\partial F}{\partial x_1}\right\}$ is a basis of $J_{d-1,e}$, the base locus of $\left|J_{d-1,e\mid Z}\right|$ consists of the points $p\in X=\mP^1\times\mP^1$ such that
        \begin{equation} \label{eq:proof:thm:J-d-e-1}
            \frac{\partial F}{\partial x_0}\left(p\right)=\frac{\partial F}{\partial x_1}\left(p\right)=F\left(p\right)=0,
        \end{equation}
        where the third condition $F\left(p\right)=0$ actually follows from the other two because of the Euler identity with respect to $x_0,x_1$.

        Suppose without loss of generality that $p=\left(\left[1:0\right],\left[1:0\right]\right)$ is a base point. Taking local affine coordinates $\left(z,w\right)=\left(\frac{x_1}{x_0},\frac{y_1}{y_0}\right)$, $p$ corresponds to the origin $\left(0,0\right)$, and the curve $Z$ is locally defined by the equation $f\left(z,w\right):=F\left(1,z,1,w\right)=0$. Being a base point of $J_{d-1,e}$ implies in particular that $\frac{\partial f}{\partial z}\left(p\right)=\frac{\partial F}{\partial x_1}\left(p\right)=0$. Since $Z$ is smooth, it must hold $\frac{\partial f}{\partial w}\left(p\right)\neq 0$, and thus the tangent line to $Z$ at $p$ has equation
        $$0=\frac{\partial f}{\partial z}\left(p\right)z+\frac{\partial f}{\partial w}\left(p\right)w=\frac{\partial f}{\partial w}\left(p\right)w,$$
        or equivalently $w=0$, which shows that $p$ is a ramification point of the projection onto the second factor (locally given by $\left(z,w\right)\mapsto w$).

        One can even check that the multiplicity as a base point agrees with the multiplicity in the ramification divisor. Indeed, note first that on $Z$ we have $x_0\frac{\partial F}{\partial x_0}+x_1\frac{\partial F}{\partial x_1}\equiv 0$, i.e.
        \begin{equation} \label{eq:proof:thm:J-d-e-2}
            \frac{\partial F}{\partial x_0}_{\mid Z}=-\frac{x_1}{x_0}\frac{\partial F}{\partial x_1}_{\mid Z}=-z\frac{\partial F}{\partial x_1}_{\mid Z}.
        \end{equation}
        Hence, the multiplicity of $p$ as a base point of $\left|J_{d-1,e\mid Z}\right|$ coincides with the vanishing order of $\frac{\partial F}{\partial x_1}_{\mid Z}$ at $p$. Using $z$ as a local coordinate for $Z$ around $p$ and expressing $w_{\mid Z}$ as a function $w\left(z\right)$, we can write (locally around $p$)
        \begin{equation} \label{eq:proof:thm:J-d-e-3}
            \frac{\partial F}{\partial x_1}_{\mid Z}=\frac{\partial F}{\partial x_1}\left(1,z,1,w\left(z\right)\right)=\frac{\partial f}{\partial z}\left(z,w\left(z\right)\right).
        \end{equation}
        From $f\left(z,w\left(z\right)\right)\equiv 0$ it follows
        \begin{equation} \label{eq:proof:thm:J-d-e-4-1}
            \frac{\partial f}{\partial z}\left(z,w\left(z\right)\right)=-\frac{\partial f}{\partial w}\left(z,w\left(z\right)\right)w'\left(z\right).
        \end{equation}
        Since $\frac{\partial f}{\partial w}\left(p\right)\neq 0$, we have
        \begin{equation} \label{eq:proof:thm:J-d-e-4-5}
            \ord_{z=0}\frac{\partial f}{\partial z}\left(z,w\left(z\right)\right)=\ord_{z=0}w'\left(z\right).
        \end{equation}
        By the above discussion, the left hand side is the multiplicity of $p$ in the base locus of $\left|J_{d-1,e\mid Z}\right|$, while the right hand side is the order of $p$ in the ramification divisor of the projection onto the second coordinate, as claimed.

        The second assertion follows easily from the Euler relation with respect to $x_0,x_1$: outside the base locus, the morphism induced by $J_{d-1,e\mid Z}$ is given given for example by
        \begin{equation} \label{eq:proof:thm:J-d-e-5}
            p\mapsto \left[\frac{\partial F}{\partial x_1}\left(p\right):-\frac{\partial F}{\partial x_0}\left(p\right)\right]=\left[x_0\left(p\right):x_1\left(p\right)\right].
        \end{equation}
        \item The restriction $J_{d,e\mid Z}$ is generated by the restrictions of $x_i\frac{\partial F}{\partial x_j}$ and $y_i\frac{\partial F}{\partial y_j}$ (Lemma \ref{lem:J-P1-P1}).

        Suppose $p\in Z$ is a base point, so that $x_i\left(p\right)\frac{\partial F}{\partial x_j}\left(p\right)=y_i\left(p\right)\frac{\partial F}{\partial y_j}\left(p\right)=0$ for all $i,j$.

        It holds $x_i\left(p\right)\neq 0$ for at least one $i$, so that $\frac{\partial F}{\partial x_j}\left(p\right)=0$ for both $j=0,1$. This implies that $p$ is a ramification point for the second projection $Z\to\mP^1$. Considering the $y_i$'s one can analogously show that $p$ is also a ramification point for the first projection. This is only possible if $Z$ is not smooth at $p$, which gives a contradiction.

        It remains to show that $J_{d,e\mid Z}$ induces a birational morphism. Considering the above set of $8$ generators, we can describe the induced morphism $\varphi$ as
        $$p\mapsto\left[\ldots\colon a_{ij}\colon\ldots\colon b_{ij}\colon\ldots\right]:=\left[\ldots\colon x_i\left(p\right)\frac{\partial F}{\partial x_j}\left(p\right)\colon\ldots\colon y_i\left(p\right)\frac{\partial F}{\partial y_j}\left(p\right)\colon\ldots\right]\in\mP^7.$$
        Note that, since the generators are not linearly independent (because of the Euler relations \eqref{eq:Euler-P1xP1}), hence the image will be contained in the intersection of hyperplanes $$\left\{\frac{1}{d}\left(a_{00}+a_{11}\right)=\frac{1}{e}\left(b_{00}+b_{11}\right)=0\right\}.$$

        Let's consider now the linear subsystem generated by the $x_i\frac{\partial F}{\partial x_j}$, whose induced rational map $\psi_x\colon Z\dashrightarrow \mP^3$ coincides with the composition of $\varphi\colon Z\to\mP^7$ with  the projection onto the first coordinates $\left[a_{00}\colon\ldots\colon a_{11}\right]$. Furthermore $\psi_x$ can be factorized as
        $$\psi_x\colon Z\dashrightarrow\mP^1\times\mP^1\stackrel{\Phi}{\longra}
        \mP^3,$$
        where $\Phi$ just the Segre embedding, and the first map is given by
        $$p\mapsto\left(\left[x_0\left(p\right):x_1\left(p\right)\right],\left[\frac{\partial F}{\partial x_0}\left(p\right):\frac{\partial F}{\partial x_1}\left(p\right)\right]\right),$$
        i.e. the first map is the product of the first projection and the morphism induced by $J_{d-1,e\mid Z
        }.$ By the first part of the Theorem, the second factor coincides (up to an automorphism of $\mP^1$) with the first projection.

        Thus, with the appropriate choice of coordinates and denoting by $\Delta$ the diagonal map, we have
        $$\psi_x\colon Z\stackrel{\pi_1}{\longra} \mP^1\stackrel{\Delta}{\longra} \mP^1\times\mP^1\stackrel{\Phi}{\longra}\mP^3$$
        i.e. from the projection of $\varphi$ onto the first four coordinates we recover the first projection of $Z\hookrightarrow\mP^1\times\mP^1\to\mP^1$.

        Analogously, we can recover the second projection $\pi_2$ from the composition of $\varphi$ with the projection onto the last four coordinates $\left[b_{00}\colon\ldots\colon b_{11}\right]$.

        Altogether, we can recover the original embedding $Z\hookrightarrow\mP^1\times\mP^1$ from the morphism $\varphi$ induced by linear system $\left|J_{d,e\mid Z}\right|$, which implies that $\varphi$ must be birational.
    \end{enumerate}
\end{proof}

We also need to study up to what extent the duality maps $\lambda_{K_X^2\otimes L^2}$ and $\lambda_{K_X\otimes L^2}$ fail to be isomorphisms.

\begin{lemma} \label{lem:duality-P1xP1}
Let $L=\caO_X\left(d,e\right)$. Then the following hold:
\begin{enumerate}
    \item If $d,e\geq 3$, the duality map $\lambda_{K_X^2\otimes L^2}\colon R_{K_X^2\otimes L^2}=R_{2d-4,2e-4}\to R_L^*=R_{d,e}^*$ is an isomorphism. Thus in particular $\ker\lambda=J_{K_X^2\otimes L^2}=J_{2d-4,2e-4}$, where $\lambda$ is as in \eqref{eq:definition-lambda}.

    \item If $d=2$ and $e\geq 2$, then $\lambda_{K_X^2\otimes L^2}$ is injective.

    \item If $d,e\geq 2$, the duality map $\lambda_{K_X\otimes L^2}\colon R_{K_X\otimes L^2}=R_{2d-2,2e-2}\to R_{K_X\otimes L}=R_{d-2,e-2}^*$ is surjective and has one-dimensional kernel. More precisely, there is a short exact sequence
    $$0 \longra H^0\left(\caO_X\right)\stackrel{\cdot c_1\left(L\right)}{\longra} H^1\left(\Omega_X^1\right)\longra R_{K_X\otimes L^2} \stackrel{\lambda_{K_X\otimes L^2}}{\longra} R_{K_X\otimes L}^*\longra 0.$$
\end{enumerate}
\end{lemma}
\begin{proof}
    \begin{enumerate}
        \item The assertion follows from Proposition \ref{prop:duality-L} and $H^1\left(\Sigma_L\right)=0$, which in turn follows from the sequence $0\to\caO_X\to\Sigma_L\to T_X\to 0$.

        \item In this case, we cannot apply Proposition \ref{prop:duality-L} because $H^1\left(\Sigma_L\otimes K_X^2\otimes L\right)=H^1\left(\Sigma_L\left(-2,e-4\right)\right)\neq 0$. Nevertheless we can still try to follow that proof and obtain on the one hand
        \begin{equation} \label{eq:RK2L2-d2}
            R_{K_X^2\otimes L^2}\cong \ker\left(H^1\left(\caE_F\otimes K_X^2\otimes L\right)\to H^1\left(\Sigma_L\otimes K_X^2\otimes L\right)\right).
        \end{equation}
        On the other hand for these degrees it still holds $H^1\left(\Sigma_L\right)=0$, which gives
        \begin{equation} \label{eq:RL-d2}
            R_L\cong \ker\left(H^1\left(\caE_F\right)\to H^1\left(\Sigma_L\right)\right)=H^1\left(\caE_F\right),
        \end{equation}
        and thus $R_L^*\cong H^1\left(\caE_F^{\vee}\otimes K_X\right)=H^1\left(\caE_F\otimes K_X^2\otimes L\right)$. Since both isomorphisms in \eqref{eq:RK2L2-d2} and \eqref{eq:RL-d2} are given by product with the class $\delta F\in\Ext^1_{\caO_X}\left(L,\caE_F\right)$, it turns out that $\lambda_{K_X^2\otimes L^2}$ is simply the inclusion of the kernel of $H^1\left(\caE_F\otimes K_X^2\otimes L\right)\to H^1\left(\Sigma_L\otimes K_X^2\otimes L\right)$ into $H^1\left(\caE_F\otimes K_X^2\otimes L\right)$.

        \item The assertion follows directly from Proposition \ref{prop:duality-K+L} and the fact that
        $$H^1\left(\Omega_X^1\right)\cong H^1\left(\caO_X\left(-2,0\right)\right)\oplus H^1\left(\caO_X\left(0,-2\right)\right) \cong\mC^2.$$
    \end{enumerate}
\end{proof}

We are now almost ready to prove the announced upper bound for the dimension of $D\subseteq\mP\left(J_{2d-4,2e-4}\right)$. To this aim, we adapt some ideas from \cite{FP-plane-curves} to construct a sufficiently high-dimensional linear subspace in $\mP\left(J_{2d-4,2e-4}\right)$ disjoint from $D$. \ToDo{}

For a fixed $G\in J_L\subseteq S_{d,e}$ with $d,e\geq 3$ consider the following map:
\begin{eqnarray} \label{eq:mu-P1xP1}
\mu_G\colon S_{d-4,e-4}\oplus S_{d-4,e-4} & \longrightarrow & S_{2d-4,2e-4} \\
\left(A,B\right) & \mapsto & AF+BG.
\end{eqnarray}

Note that, since both $F,G\in J_L$, it holds $\im\mu\subseteq J_{K_X^2\otimes L^2}=J_{2d-4,2e-4}$ (for this equality we need $d,e\geq 3$).

\begin{theorem} \label{thm:upper-bound-dim-D-P1xP1}
Suppose $d,e\geq 3$. For a general $G\in J_L=J_{d,e}$ (in particular, $G\not\in\mC F$) it holds:
\begin{enumerate}
\item $\ker\mu_G=0$, and
\item $D\cap\mP\left(\im\mu_G\right)=\emptyset$.
\end{enumerate}
\end{theorem}
\begin{proof}
\begin{enumerate}
\item Suppose $\mu_G\left(A,B\right)=AF+BG=0\in S_{2d-4,2e-4}=H^0\left(K_X^2\otimes L^2\right)$. Then the restriction to $Z=\left\{F=0\right\}$ is also
$$B_{\mid Z}G_{\mid Z}\equiv 0\in H^0\left(\left(K_X^2\otimes L^2\right)_{\mid Z}\right).$$
Since $G_{\mid Z}\not\equiv 0$ (because $G$ is not a multiple of  $F$), it must hold $B_{\mid Z}\equiv 0$, i.e. $B$ is a multiple of $F$. But since $B\in S_{d-4,e-4}$ has smaller bidegree than $F$, it must be $B=0$. Thus also $FA=0\in S_{2d-4,2e-4}$, and thus $A=0$ because the map $S_{d-4,e-4}\stackrel{\cdot F}{\longra}S_{2d-4,2e-4}$ is injective.

\item Consider the set
\begin{equation} \label{eq:I-upper-bound}
I=\left\{\left(A,B,\alpha,\beta,G\right)\,\left|\,\begin{array}{l}
A,B\in S_{d-4,e-4},\alpha,\beta\in H^0\left(K_X\otimes L\right)\setminus\left\{0\right\}=S_{d-2,e-2}\setminus\left\{0\right\}\\
G\in J_{d,e}\setminus \left\langle F\right\rangle, AF+BG=\alpha\beta
\end{array}\right\}\right.
\end{equation}
which is a quasi-projective subvariety of $S_{d-4,e-4}\times S_{d-4,e-4}\times S_{d-2,e-2}\times S_{d-2,e-2}\times J_{d,e}$.

By definition of $D$, it holds
$$D\cap\mP\left(\im\mu_G\right)\neq\emptyset$$
if and only if $G\in\pi\left(I\right)$, where $\pi\colon I\to J_{d,e}\setminus \mC F$ denotes the projection $\left(A,B,\alpha,\beta,G\right)\mapsto G$.

We want thus to show that $\pi$ is not dominant, for which we can replace $I$ by its (open, dense) smooth subset. Assume from now on that $\pi$ is dominant.

By Theorem \ref{thm:J-d-e}, the restriction of the Jacobian $J_{d,e}$ to $Z$ induces a birational map onto a non-degenerate projective curve $C\subseteq\mP^r$ for certain $r\leq\dim J_{d,e}-2$. The pull-backs of the hyperplane sections of $\mP^r$ are by construction the vanishing divisors $\left(G_{\mid Z}\right)$. Let $U$ denote the subset $\left\{G\in J_{d,e}\setminus\mC F\,\mid\,\left(G_{\mid Z}\right)\text{ is reduced}\right\}$. Given a fixed $G\in U$ with $\left(G_{\mid C}\right)=p_1+p_2+\ldots+p_n$, the fundamental group $\pi_1\left(U,G\right)$ induces a monodromy action on $\left\{p_1,\ldots,p_n\right\}$, i.e. there is a group homomorphism
\begin{equation} \label{eq:monodromy-P1-P1}
\rho\colon\pi_1\left(U,G\right)\to\Sym\left(\left\{p_1,\ldots,p_n\right\}\right).
\end{equation}
By the Uniform Position Theorem \cite[Lemma in p. 111]{ACGH}, the monodromy group is the full symmetric group, i.e. the group homomorphism \eqref{eq:monodromy-P1-P1} is surjective.

Replacing $U$ by a non-empty Zariski-open subset the fundamental group does not change, hence we can assume that $U$ is contained in $\pi\left(I\right)$ (recall that we are assuming that $\pi$ is dominant). Replacing $I$ by $\pi^{-1}\left(U\right)$ we can assume $\pi\colon I\ra U$ is a surjective morphism of smooth (quasi-projective) varieties.

Fix $G\in U$ and set $\left(G_{\mid Z}\right)=p_1+\ldots+p_n$, where $n=\deg\caO_Z\left(d,e\right)=2de>2de-2d-2e=2g\left(Z\right)-2$. By the definition of $U$, there are $A,B,\alpha,\beta$ such that $AF+BG=\alpha\beta$. The restriction to $Z=\left\{F=0\right\}$ gives then
$$B_{\mid Z}G_{\mid Z}=\alpha_{\mid Z}\beta_{\mid Z},$$
so $p_1+\ldots+p_n\leq \left(\alpha_{\mid Z}\right)+\left(\beta_{\mid Z}\right)$. Up to reordering the indices of the points and switching $\alpha$ and $\beta$, we can assume $\alpha_{\mid Z}\left(p_i\right)=0$ for $i=1,\ldots,m\geq\frac{n}{2}$, and $\alpha_{\mid Z}\left(p_i\right)\neq 0$ for $i=m+1,\ldots,n$. Moreover since $\alpha_{\mid Z}\in H^0\left(K_Z\right)$ and thus $\deg\left(\alpha_{\mid Z}\right)=2g\left(Z\right)-2<n$, it holds $m<n$.

Now for every $i=1,\ldots,m$ there is a closed path $\gamma_i$ on $U$ based at $G$ whose monodromy action permutes $p_i$ and $p_{m+1}$ and leaves all other points fixed. We can lift this path to an open path on $I$, starting on $\left(A,B,\alpha,\beta,G\right)$, i.e. we can find a continuous map
$$\tilde{\gamma_i}=\left(\tilde{A},\tilde{B},\tilde{\alpha},\tilde{\beta},\tilde{G}\right)\colon\left[0,1\right]\to I$$
with $\tilde{\gamma_i}\left(0\right)=\left(A,B,\alpha,\beta,G\right)$ and $\tilde{A}\left(t\right)F+\tilde{B}\left(t\right)\tilde{G}\left(t\right)=\tilde{\alpha}\left(t\right)\tilde{\beta}\left(t\right)$ for every $t\in\left[0,1\right]$.

Set $\alpha_i:=\tilde{\alpha}\left(1\right),\beta_i:=\tilde{\beta}\left(1\right), A_i:=\tilde{A}\left(1\right)$ and $B_i:=\tilde{B}\left(1\right)$, which satisfy $A_iF+B_iG=\alpha_i\beta_i$. Thus restricting to $Z$ we obtain
$$p_1+\ldots+p_n=\left(G_{\mid Z}\right)\leq \left({B_iG}_{\mid Z}\right)=\left(\alpha_{i\mid Z}\right)+\left(\beta_{i\mid Z}\right).$$
Moreover, it holds $\alpha_{i\mid Z}\left(p_{m+1}\right)=\alpha_{\mid Z}\left(p_i\right)=0$ and $\alpha_{i\mid Z}\left(p_j\right)=\alpha_{\mid Z}\left(p_j\right)=0$ for every $j=1,\ldots,m$, $j\neq i$, while $\alpha_{i\mid Z}\left(p_i\right)=\alpha_{\mid Z}\left(p_{m+1}\right)\neq 0$.

In this way, we obtain $\alpha_0=\alpha,\alpha_1,\ldots,\alpha_m$, whose restrictions to $Z$ are linearly independent holomorphic $1$-forms, and hence $m+1\leq g\left(Z\right)$. But this is a contradiciton, since
$$n\leq 2m \leq 2g\left(Z\right)-2<n.$$

\end{enumerate}
\end{proof}

As announced, we can use the previous theorem to find an upper bound on $\dim D$.

\begin{corollary} \label{cor:upper-bound-dim-D-P1-P1}
If $d,e\geq 3$ it holds
\begin{align} \label{eq:upper-bound-dim-D-P1-P1}
\dim D\leq &\, de-d-e-4.
\end{align}
\end{corollary}
\begin{proof}
Let $G$ be general, so that $D\cap\mP\left(\im\mu_G\right)=\emptyset\subseteq\mP\left(J_{2d-4,2e-4}\right)$. Then it must hold $\dim D+\dim\mP\left(\im\mu_G\right)<\dim\mP\left(J_{2d-4,2e-4}\right)$, and hence
\begin{multline}
\dim D<\dim J_{2d-4,2e-4}-\dim\im\mu_G=\left(\dim S_{2d-4,2e-4}-\dim R_{2d-4,2e-4}\right)-2\dim S_{d-4,e-4}=\\
\stackrel{(*)}{=}\left(2d-3\right)\left(2e-3\right)-\dim R_{d,e}-2\left(d-3\right)\left(e-3\right)=2de-9-\left(\dim S_{d,e}-\dim J_{d,e}\right)=\\
=2de-9-\left(d+1\right)\left(e+1\right)+\dim J_{d,e}=de-d-e-10+\dim J_{d,e}\leq\\
\leq de-d-e-10+h^0\left(\Sigma_L\right)=de-d-e-3.
\end{multline}
Note that in $(*)$ we have used $\dim R_{2d-4,2e-4}=\dim R_{d,e}$, i.e. $\dim R_{K_X^2\otimes L^2}=\dim R_L$, which follows from $d,e\geq 3$ and Lemma \ref{lem:duality-P1xP1}.
\end{proof}

We are now ready to prove the final result of the paper.

\begin{theorem} \label{thm:IVHS-maximal-P1-P1}
For any $d,e\geq 1$, $L=\caO_X\left(d,e\right)$ has maximal infinitesimal variation at every smooth $Z\in\left|L\right|$.
\end{theorem}
\begin{proof}
Let $Z\in\left|L\right|$ be a smooth curve defined by a polynomial $F\in H^0\left(L\right)=S_{d,e}$. Let's consider first the cases of smaller degrees.

If $d=1$ or $e=1$, $Z$ is a rational curve, hence the Hodge structure is trivial and there is nothing to prove.

Suppose now that $d=2$ and $e\geq 2$. Since $H^0\left(K_X\right)=0$, it holds $R_{K_X\otimes L}=H^0\left(K_X\otimes L\right)=S_{0,e-2}$. Moreover it holds $J_{K_X^2\otimes L^2}=J_{0,2e-4}=0$ because the Jacobian ideal is generated in degrees $\left(d-1,e\right)=\left(1,e\right)$ and $\left(d,e-1\right)=\left(2,e-1\right)$. This means that $R_{K_X^2\otimes L^2}=H^0\left(K_X^2\otimes L^2\right)=S_{0,2e-4}$.

Suppose that $\alpha\in\ker\varphi_{K_X\otimes L}\left(\tau\right)$ for a general $\tau\in H^0\left(L\right)=S_{2,e}$, so that $\lambda_{K_X^2\otimes L^2}\left(\alpha^2\right)=0\in R_L^*$ by Lemma \ref{lem:alpha2}. But the injectivity of $\lambda_{K_X^2\otimes L^2}$ (Lemma \ref{lem:duality-P1xP1}) implies $\alpha^2=0\in S_{0,2e-4}$, and thus $\alpha=0$.

We can now consider the case $d,e\geq 3$. If $L$ does not have maximal infinitesimal variation at $Z$, then Theorem \ref{thm:lower-bound-dim-D} implies
\begin{equation} \label{eq:lower-bound-dim-D-P1-P1}
\dim D\geq h^0\left(K_X\otimes L\right)+p_g\left(X\right)-1=h^0\left(\caO_X\left(d-2,e-2\right)\right)-1=\left(d-1\right)\left(e-1\right)-1=de-d-e,
\end{equation}
which contradicts Corollary \ref{cor:upper-bound-dim-D-P1-P1}.
\end{proof}

\begin{remark} \label{rmk:thm-upper-bound-more-general}
Note that a very important point in the proof of Theorem \ref{thm:upper-bound-dim-D-P1xP1} is that the generalized Jacobian ideal in degree precisely $L$ (i.e. the image of $\widetilde{d\sigma}\colon H^0\left(T_X\right)\to H^0\left(L_{\mid Z}\right)$) induces a birational embedding of $Z$ in a projective space. This property might also hold for more general surfaces, giving some hope that this approach might work in a few more cases.
\end{remark}

\bibliographystyle{alpha}

\end{document}